\documentclass[11pt, final]{article}
\usepackage{a4}
\usepackage{amsmath}%
\usepackage{amstext}%
\usepackage{amssymb}%
\usepackage{showkeys}%
\usepackage{epsfig}%
\usepackage{cite}

\usepackage{listings}
\newtheorem{theorem}{Theorem}

\newtheorem{algorithm}[theorem]{Algorithm}

\newtheorem{rem}[theorem]{Remark}
\newenvironment{remark}{\rem\rm}{\endrem}

\newcounter{unnumber}

\newtheorem{prf}[unnumber]{Proof.}
\newenvironment{proof}{\prf\rm}{\hfill{$\blacksquare$}\endprf}
\newcommand{\R}{\mathbb{R}}%
\newcommand{\N}{\mathbb{N}}%
\newcommand{\ol}{\overline}%
\DeclareMathOperator*\inte{int}%
\DeclareMathOperator*\sqri{sqri}%
\DeclareMathOperator*\dom{dom}%
\DeclareMathOperator*\gr{Gr}%
\DeclareMathOperator*\id{Id}%
\DeclareMathOperator*\prox{prox}%
\DeclareMathOperator*\argmin{argmin}
\DeclareMathOperator*\sn{sgn}

\title{A primal-dual splitting algorithm for finding zeros of sums of maximally monotone operators}

\author{Radu Ioan Bo\c{t} \thanks{Department of Mathematics, Chemnitz University of Technology,
D-09107 Chemnitz, Germany, e-mail:
 radu.bot@mathematik.tu-chemnitz.de. Research partially supported by DFG (German Research Foundation), projects BO 2516/4-1 and WA922/1-3.} \and Ern\"{o} Robert Csetnek \thanks {Department of Mathematics, Chemnitz University of Technology, D-09107 Chemnitz, Germany, e-mail:
 robert.csetnek@mathematik.tu-chemnitz.de. Research supported by DFG (German Research Foundation), project BO 2516/4-1.} \and Andr\'e Heinrich
\thanks{Department of Mathematics, Chemnitz University of Technology, D-09107 Chemnitz, Germany, e-mail:
 andre.heinrich@mathematik.tu-chemnitz.de. Research supported by the European Union, the European Social Fund (ESF) and prudsys AG in Chemnitz.}}

\begin{document}
\maketitle

\noindent \textbf{Abstract.} We consider the primal problem of finding the zeros of the sum of a maximally monotone operator with the composition of another maximally monotone operator with a linear continuous operator and a corresponding dual problem formulated by means of the inverse operators. A primal-dual splitting algorithm which simultaneously solves the two problems in finite-dimensional spaces is presented. The scheme uses at each iteration separately the resolvents of the maximally monotone operators involved and it gives rise to a splitting algorithm for finding the zeros of the sum of compositions of maximally monotone operators with linear continuous operators. The iterative schemes are used for solving nondifferentiable convex optimization problems arising in image processing and in location theory.\vspace{1ex}

\noindent \textbf{Key Words.} maximally monotone operator, resolvent, operator splitting, subdifferential,
minimization algorithm, duality\vspace{1ex}

\noindent \textbf{AMS subject classification.} 47H05, 65K05, 90C25

\section{Introduction and preliminaries}\label{sec1}
In this paper we propose an iterative scheme for solving the inclusion problem
$$\mbox{find} \ x \in X \ \mbox{such that} \ 0 \in Ax + K^*BKx,$$
where $X$ and $Y$ are Hilbert spaces, $A:X \rightrightarrows X$ and $B:Y\rightrightarrows Y$ are maximally monotone operators and $K:X \rightarrow Y$ is a linear continuous operator, which makes separately use of the resolvents of $A$ and $B$. The necessity of having such an algorithm is given by the fact that the classical splitting algorithms have considerable limitations when employed on the inclusion problem under investigation in its whole generality. Indeed, the \textit{forward-backward algorithm} (see \cite{bauschke-book}) is a valuable option in this sense when $B$ is single-valued and cocoercive, while the use of \textit{Tseng's algorithm} (see \cite{tseng}) asks for $B$ being single-valued and Lipschitz continuous on a superset of the image of the domain of $A$ through $K$. On the other hand, the \textit{Douglas-Rachford algorithm} (see \cite{bauschke-book, dor}) asks for the maximal monotonicity of $A$ and $K^*BK$ and employs the resolvent of the latter, which can be expressed by means of the resolvent of $B$ only in some very exceptional situations (see \cite[Proposition 23.23]{bauschke-book}).

The aim of this article is to overcome this shortcoming by providing a \textit{primal-dual splitting algorithm} for simultaneously solving this inclusion problem and its dual inclusion problem in the sense of Attouch-Th\'era (see \cite{att-th, bauschke-bot, pennanen}), in the formulation of which the resolvents of $A$ and $B$ appear separately. In the case when $A$ and $B$ are subdifferentials of proper, convex and lower semicontinuous functions we rediscover as particular case the iterative method from \cite{ch-pck}. We also show how the provided primal-dual algorithm gives rise to a primal-dual iterative method for finding the zeros of the sum of compositions of maximally monotone operators with linear continuous operators. The latter will find application when solving nondifferentiable convex optimization problems arising in image processing and in location theory having in the objective the sum of (more than two) compositions of proper, convex and lower semicontinuous functions with linear continuous operators.

For another primal-dual splitting algorithm for simultaneously solving a primal inclusion problem and its Attouch-Th\'era-type dual inclusion problem, recently introduced in the literature, we refer the reader to \cite{br-combettes, combettes-pesquet}. By using a consecrated \textit{product space approach}, this method basically reformulates the primal-dual pair as the problem of finding the zeros of the sum of a maximally monotone operator and a monotone and Lipschitz continuous operator, which is then solved by making use of the relaxed version of Tseng's algorithm.

The structure of the paper is the following. The remaining of this section is dedicated to some elements of convex analysis and of the theory of maximally monotone operators. In Section \ref{sec2} we motivate and formulate the primal-dual splitting algorithm for solving the problem of finding the zeros of the sum of a maximally monotone operator with the composition of another maximally monotone operator with a linear continuous operator and its dual problem and investigate its convergence properties. In Section \ref{sec3} we formulate a primal-dual splitting algorithm for the problem of finding the zeros of the sum of compositions of maximally monotone operators with linear continuous operators, while in Section \ref{sec4} we employ the two primal-dual schemes for solving several classes of nondifferentiable convex optimization problems. Finally, we consider applications of the presented algorithms in image deblurring and denoising and when solving the Fermat-Weber location problem. For the latter we compare their performances to the ones of some iterative schemes recently introduced in the literature.

In what follows we recall some elements of convex analysis and of the theory of maximally monotone operators in Hilbert spaces and refer the reader in this respect to the books \cite{b-hab, bauschke-book, EkTem, simons, Zal-carte}.

Let $X$ be a real Hilbert space with \textit{inner product} $\langle\cdot,\cdot\rangle$ and associated \textit{norm} $\|\cdot\|=\sqrt{\langle \cdot,\cdot\rangle}$. For a function $f:X\rightarrow\overline{\R}$, where $\overline{\R}:=\R\cup\{\pm\infty\}$ is the extended real line, we denote by $\dom f=\{x\in X:f(x)<+\infty\}$ its \emph{effective domain} and say that $f$ is \emph{proper} if $\dom f\neq\emptyset$ and $f(x)>-\infty$ for all $x\in X$.  Let $f^*:X \rightarrow \overline \R$, $f^*(u)=\sup_{x\in X}\{\langle u,x\rangle-f(x)\}$ for all $u\in X$, be the \emph{conjugate function} of $f$. The \emph{subdifferential} of $f$ at $x\in f^{-1}(\R)$ is the set $\partial f(x):=\{u\in X:f(y)\geq f(x)+\langle u,y-x\rangle \ \forall y\in X\}$. We take by convention $\partial f(x):=\emptyset$, if $x\notin f^{-1}(\R)$. For every $\gamma >0$ and every $x \in X$ it holds $\partial(\gamma f)(x) = \gamma \partial f(x)$. When $Y$ is another Hilbert space and $K:X \rightarrow Y$ a linear continuous operator, then $K^* : Y \rightarrow X$, defined by $\langle K^*y,x\rangle = \langle y,Kx \rangle$ for all $(x,y) \in X \times Y$, denotes the \textit{adjoint operator} of $K$.

Let $C\subseteq X$ be a nonempty set. The \emph{indicator function} of $C$, $\delta_C:X\rightarrow \overline{\R}$, is the function which takes the value $0$ on $C$ and $+\infty$ otherwise. The subdifferential of the indicator function is the \emph{normal cone} of $C$, that is $N_C(x)=\{u\in X:\langle u,y-x\rangle\leq 0 \ \forall y\in C\}$, if $x\in C$ and $N_C(x)=\emptyset$ for $x\notin C$. If $C$ is a convex set, we denote by $\sqri C:=\{x\in C:\cup_{\lambda>0}\lambda(C-x) \ \mbox{is a closed linear subspace of} \ X\}$ its \emph{strong quasi-relative interior}. The strong quasi-relative interior of $C$ is a superset of the topological interior of $C$, i.e., $\inte C\subseteq\sqri C$ (in general this inclusion may be strict). If $X$ is finite-dimensional, than $\sqri C$ coincides with the \emph{relative interior} of $C$, which is the interior of $C$ with respect to the affine hull of this set. For more results relating to generalized interiority-type notions we refer the reader to \cite{bot-csetnek, Zal-carte, bauschke-book, simons, b-hab}.

For an arbitrary set-valued operator $A:X\rightrightarrows X$ we denote by $\gr A=\{(x,u)\in X\times X:u\in Ax\}$ its \emph{graph}, by $\dom A=\{x \in X : Ax \neq \emptyset\}$ its \emph{domain} and by $A^{-1}:X\rightrightarrows X$ its \emph{inverse operator}, defined by $(u,x)\in\gr A^{-1}$ if and only if $(x,u)\in\gr A$.  We say that $A$ is \emph{monotone} if $\langle x-y,u-v\rangle\geq 0$ for all $(x,u),(y,v)\in\gr A$. A monotone operator $A$ is said to be \emph{maximally monotone}, if there exists no proper monotone extension of the graph of $A$ on $X\times X$. Notice that the subdifferential of a proper, convex and lower semicontinuous function is a maximally monotone operator (cf. \cite{rock}). A single-valued linear operator $A:X \rightarrow X$ is said to be \textit{skew}, if $\langle x,Ax\rangle =0$ for all $x \in X$. The skew operators are maximally monotone and if they are not identical to zero and the dimension of $X$ is greater than or equal to $2$, then they fail to be subdifferentials (see \cite{simons}). When $f:X \rightarrow \overline \R$ is a proper, convex and lower semicontinuous it holds $(\partial f)^{-1} = \partial f^*$.  The \emph{resolvent} of $A$, $J_A:X \rightrightarrows X$, is defined by $J_A=(\id_X+A)^{-1}$, where $\id_X :X \rightarrow X, \id_X(x) = x$ for all $x \in X$, is the \textit{identity operator} on $X$. Moreover, if $A$ is maximally monotone, then $J_A:X \rightarrow X$ is single-valued and maximally monotone (cf. \cite[Proposition 23.7 and Corollary 23.10]{bauschke-book}). For an arbitrary $\gamma>0$ we have (see \cite[Proposition 23.2]{bauschke-book})
$$p\in J_{\gamma A}x \ \mbox{if and only if} \ (p,\gamma^{-1}(x-p))\in\gr A$$
and (see \cite[Proposition 23.18]{bauschke-book})
\begin{equation}\label{j-inv-op}
J_{\gamma A}+\gamma J_{\gamma^{-1}A^{-1}}\circ \gamma^{-1}\id\nolimits_X=\id\nolimits_X.
\end{equation}
When $f:X \rightarrow \overline \R$ is a proper, convex and lower semicontinuous function and $\gamma > 0$, for every $x \in X$ we denote by $\prox_{\gamma f}(x)$ the \textit{proximal point} of parameter $\gamma$ of $f$ at $x$, which is the unique optimal solution of the optimization problem
\begin{equation}\label{prox-def}\inf_{y\in X}\left \{f(y)+\frac{1}{2\gamma}\|y-x\|^2\right\}.
\end{equation}
Notice that $J_{\gamma\partial f}=(\id_X+\gamma\partial f)^{-1}=\prox_{\gamma f}$, thus  $\prox_{\gamma f} :X \rightarrow X$ is a single-valued operator fulfilling the extended\textit{ Moreau's decomposition formula}
\begin{equation}\label{prox-f-star}
\prox\nolimits_{\gamma f}+\gamma\prox\nolimits_{(1/\gamma)f^*}\circ\gamma^{-1}\id\nolimits_X=\id\nolimits_X.
\end{equation}
Let us also recall that the function $f:X \rightarrow \overline \R$ is said to be \emph{strongly convex} (with modulus $\gamma >0$), if $f-\frac{\gamma}{2}\|\cdot\|^2$ is a convex function.

Finally, we notice that for $f=\delta_C$, where $C\subseteq X$ is a nonempty closed and convex set, it holds
\begin{equation}\label{projection}
J_{\gamma N_C}=J_{N_C}=J_{\partial \delta_C} = (\id\nolimits_X+N_C)^{-1}=\prox\nolimits_{\delta_C}=P_C,
\end{equation}
where  $P_C :X \rightarrow C$ denotes the \emph{projection operator} on $C$ (see \cite[Example 23.3 and Example 23.4]{bauschke-book}).

\section{A primal-dual splitting algorithm for finding the zeros of $A + K^*BK$}\label{sec2}

For $X$ and $Y$ real Hilbert spaces, $A:X\rightrightarrows X$ and $B:Y\rightrightarrows Y$ maximally monotone operators and $K:X\rightarrow Y$ a linear and continuous operator we consider the problem of finding the pairs $(\widehat{x},\widehat{y}) \in X \times Y$ fulfilling the system of inclusions
\begin{equation}\label{syst-incl-2}
Kx\in B^{-1}y \ \mbox{and} \ -K^*y\in Ax.
\end{equation}
If $(\widehat{x},\widehat{y})$ fulfills \eqref{syst-incl-2}, then $\widehat x$ is a solution of the primal inclusion problem
\begin{equation}\label{ecu1}
\mbox{find} \ x \in X \ \mbox{such that} \  0\in Ax + K^*BKx
\end{equation}
and $\widehat y$ is a solution of its dual inclusion problem in the sense of Attouch-Th\'era
\begin{equation}\label{ecu2}
\mbox{find} \ y \in Y \ \mbox{such that} \ 0\in B^{-1}y-KA^{-1}(-K^*)y.
\end{equation}
On the other hand, if $\widehat x \in X$ is a solution of the problem \eqref{ecu1}, then there exists a solution $\widehat y$ of \eqref{ecu2} such that  $(\widehat{x},\widehat{y})$ fulfills \eqref{syst-incl-2}, while, if $\widehat y \in Y$ is a solution of the problem \eqref{ecu2}, then there exists a solution $\widehat x$ of \eqref{ecu1} such that $(\widehat{x},\widehat{y})$ fulfills \eqref{syst-incl-2}. We refer the reader to \cite{alduncin, att-th, bauschke-bot, br-combettes, eckstein-ferris, pennanen} for more algorithmic and theoretical aspects concerning the primal-dual pair of inclusion problems
\eqref{ecu1}-\eqref{ecu2}.

For all $\sigma, \tau > 0$  it  holds
\begin{equation}\label{motiv}
\begin{aligned}
(\widehat{x},\widehat{y}) \ \mbox{is a solution of} \  \eqref{syst-incl-2} \Leftrightarrow \widehat y + \sigma K\widehat x \in & \ (\id\nolimits_Y + \sigma B^{-1})\widehat y \ \mbox{and} \ \widehat x - \tau K^*\widehat y \in (\id\nolimits_X + \tau A)\widehat x\\
\Leftrightarrow \widehat y = J_{\sigma B^{-1}}(\widehat y + \sigma K\widehat x) & \ \mbox{and}  \ \widehat x \in J_{\tau A}(\widehat x - \tau K^*\widehat y).
\end{aligned}
\end{equation}
The above equivalences motivate the following algorithm for solving \eqref{syst-incl-2}.

\begin{algorithm}\label{alg-2-op}$ $

\noindent\begin{tabular}{rl}
\verb"Initialization": & \verb"Choose" $\sigma,\tau>0$ \verb"such that" $\sigma\tau\|K\|^2<1$ \verb"and" $(x^0,y^0)\in X\times Y$.\\
& \verb"Set" $\ol{x}^0:=x^0$.\\
\verb"For" $n\geq 0$ \verb"set": & $y^{n+1}:=J_{\sigma B^{-1}}(y^n+\sigma K\ol{x}^n)$\\
& $x^{n+1}:=J_{\tau A}(x^n-\tau K^*y^{n+1})$\\
& $\ol{x}^{n+1}:=2x^{n+1}-x^n$
\end{tabular}

\end{algorithm}

\begin{theorem}\label{th-2-op} Assume that the system of inclusions \eqref{syst-incl-2} has a solution $(\widehat{x}, \widehat{y})\in X\times Y$  and  let $(x^n,\ol{x}^n,y^n)_{n \geq 0}$ be the sequence generated by Algorithm \ref{alg-2-op}. The following statements are true:

(i) For any $n \geq 0$ it holds
\begin{equation}\label{bound-seq} \frac{\|x^n-\widehat{x}\|^2}{2\tau}+(1-\sigma\tau\|K\|^2)\frac{\|y^n-\widehat{y}\|^2}{2\sigma}\leq
\frac{\|x^0-\widehat{x}\|^2}{2\tau}+\frac{\|y^0-\widehat{y}\|^2}{2\sigma},
\end{equation}
thus the sequence $(x^n,y^n)_{n \geq 0}$ is bounded.

(ii) If $X$ and $Y$ are finite-dimensional, then the sequence $(x^n,y^n)_{n \geq 0}$ converges to a solution of the system of inclusions \eqref{syst-incl-2}.
\end{theorem}

\begin{proof} (i) For any $n \geq 0$ the iterations of  Algorithm \ref{alg-2-op} yield that
\begin{equation}\label{gr-A}
\left(x^{n+1},\frac{1}{\tau}(x^n-\tau K^*y^{n+1}-x^{n+1})\right)\in\gr A,
\end{equation}
hence the monotonicity of $A$ implies
\begin{equation}\label{mon-A}
0\leq \left \langle x^{n+1}-\widehat{x},\frac{1}{\tau}(x^n-x^{n+1})-K^*y^{n+1}+K^*\widehat{y}\right \rangle.
\end{equation}

Similarly, for any $n \geq 0$ we have
\begin{equation}\label{gr-B}\left(y^{n+1},\frac{1}{\sigma}(y^n+\sigma K\ol{x}^n-y^{n+1})\right)\in\gr B^{-1},
\end{equation} thus
\begin{equation}\label{mon-B}
0\leq \left \langle K\ol{x}^n+\frac{1}{\sigma}(y^n-y^{n+1})-K\widehat{x},y^{n+1}-\widehat{y} \right \rangle.
\end{equation}

On the other hand, for any $n \geq 0$ we have that
$$\|x^{n+1}-\widehat{x}\|^2+ \left \langle x^{n+1}-\widehat{x},\frac{1}{\tau}(x^n-x^{n+1})-K^*y^{n+1}+K^*\widehat{y} \right\rangle=$$
$$\left\langle x^{n+1}-\widehat{x},\frac{1}{\tau}(x^n-\widehat{x})+\left(1-\frac{1}{\tau}\right)(x^{n+1}-\widehat{x})-K^*y^{n+1}+K^*\widehat{y} \right \rangle=$$
$$\frac{1}{\tau}\langle x^{n+1}-\widehat{x},x^n-\widehat{x}\rangle+\left(1-\frac{1}{\tau}\right)\|x^{n+1}-\widehat{x}\|^2+
\langle K(x^{n+1}-\widehat{x}),-y^{n+1}+\widehat{y}\rangle,$$ hence
$$\frac{1}{\tau}\|x^{n+1}-\widehat{x}\|^2+ \left \langle x^{n+1}-\widehat{x},\frac{1}{\tau}(x^n-x^{n+1})-K^*y^{n+1}+K^*\widehat{y}\right\rangle=$$
$$\frac{1}{\tau}\langle x^{n+1}-\widehat{x},x^n-\widehat{x}\rangle+\langle K(x^{n+1}-\widehat{x}),-y^{n+1}+\widehat{y}\rangle=$$
$$-\frac{1}{2\tau}\|x^{n+1}-x^n\|^2+\frac{1}{2\tau}\|x^{n+1}-\widehat{x}\|^2+\frac{1}{2\tau}\|\widehat{x}-x^n\|^2+
\langle K(x^{n+1}-\widehat{x}),-y^{n+1}+\widehat{y}\rangle,$$ where, for deriving the last formula,
we use the identity
$$\langle a,b\rangle=-\frac{1}{2}\|a-b\|^2+\frac{1}{2}\|a\|^2+\frac{1}{2}\|b\|^2 \ \forall a,b\in X.$$
Consequently, for any $n \geq 0$ it holds
\begin{equation}\label{eq1}
\begin{aligned}
\frac{1}{2\tau}\|x^{n+1}-\widehat{x}\|^2+ \left \langle x^{n+1}-\widehat{x},\frac{1}{\tau}(x^n-x^{n+1})-K^*y^{n+1}+K^*\widehat{y} \right \rangle=\\
-\frac{1}{2\tau}\|x^{n+1}-x^n\|^2+\frac{1}{2\tau}\|x^n-\widehat{x}\|^2+ \langle K(x^{n+1}-\widehat{x}),-y^{n+1}+\widehat{y}\rangle.
\end{aligned}
\end{equation}
Thus, by combining \eqref{mon-A} and \eqref{eq1}, we get for any $n \geq 0$
\begin{equation}\label{ineq-x}\frac{1}{2\tau}\|x^{n+1}-\widehat{x}\|^2\leq
-\frac{1}{2\tau}\|x^{n+1}-x^n\|^2+\frac{1}{2\tau}\|x^n-\widehat{x}\|^2+
\langle K(x^{n+1}-\widehat{x}),-y^{n+1}+\widehat{y}\rangle.\end{equation}

By proceeding in analogous manner we obtain the following estimate for any $n \geq 0$
$$\|y^{n+1}-\widehat{y}\|^2+\left\langle K\ol{x}^n+\frac{1}{\sigma}(y^n-y^{n+1})-K\widehat{x},y^{n+1}-\widehat{y}\right\rangle=$$
$$\left \langle \frac{1}{\sigma}(y^n-\widehat{y})+\left(1-\frac{1}{\sigma}\right)(y^{n+1}-\widehat{y})+K\ol{x}^n-K\widehat{x}, y^{n+1}-\widehat{y}\right\rangle=$$
$$\frac{1}{\sigma}\langle y^n-\widehat{y},y^{n+1}-\widehat{y}\rangle+\left(1-\frac{1}{\sigma}\right)\|y^{n+1}-\widehat{y}\|^2 + \langle K(\ol{x}^n-\widehat{x}),y^{n+1}-\widehat{y}\rangle,$$
hence
$$\frac{1}{\sigma}\|y^{n+1}-\widehat{y}\|^2+\left\langle K\ol{x}^n+\frac{1}{\sigma}(y^n-y^{n+1})-K\widehat{x},y^{n+1}-\widehat{y}\right\rangle=$$
$$\frac{1}{\sigma}\langle y^n-\widehat{y},y^{n+1}-\widehat{y}\rangle+\langle K(\ol{x}^n-\widehat{x}),y^{n+1}-\widehat{y}\rangle=
$$$$-\frac{1}{2\sigma}\|y^{n+1}-y^n\|^2+\frac{1}{2\sigma}\|y^{n+1}-\widehat{y}\|^2+\frac{1}{2\sigma}\|\widehat{y}-y^n\|^2+\langle K(\ol{x}^n-\widehat{x}),y^{n+1}-\widehat{y}\rangle.$$
From here we obtain for any $n \geq 0$
\begin{equation}\label{eq2}
\begin{aligned}
\frac{1}{2\sigma}\|y^{n+1}-\widehat{y}\|^2+\left\langle K\ol{x}^n+\frac{1}{\sigma}(y^n-y^{n+1})-K\widehat{x},y^{n+1}-\widehat{y}\right\rangle=\\
-\frac{1}{2\sigma}\|y^{n+1}-y^n\|^2+\frac{1}{2\sigma}\|y^n-\widehat{y}\|^2+
\langle K(\ol{x}^n-\widehat{x}),y^{n+1}-\widehat{y}\rangle.
\end{aligned}
\end{equation}
and, thus, by combining \eqref{mon-B} and \eqref{eq2}, it follows
\begin{equation}\label{ineq-y}\frac{1}{2\sigma}\|y^{n+1}-\widehat{y}\|^2\leq
-\frac{1}{2\sigma}\|y^{n+1}-y^n\|^2+\frac{1}{2\sigma}\|y^n-\widehat{y}\|^2+
\langle K(\ol{x}^n-\widehat{x}),y^{n+1}-\widehat{y}\rangle.
\end{equation}

Summing up the inequalities \eqref{ineq-x} and \eqref{ineq-y} and taking into account the definition of $\ol {x}^n$  we obtain for any $n \geq 0$
\begin{equation}\label{ineq-x-y-K}
\begin{aligned}
\frac{1}{2\tau}\|x^{n+1}-\widehat{x}\|^2 &+\frac{1}{2\sigma}\|y^{n+1}-\widehat{y}\|^2\leq\\
\frac{1}{2\tau}\|x^n-\widehat{x}\|^2+\frac{1}{2\sigma}\|y^n-\widehat{y}\|^2
&-\frac{1}{2\tau}\|x^{n+1}-x^n\|^2-\frac{1}{2\sigma}\|y^{n+1}-y^n\|^2+\\
\langle K(x^{n+1}+x^{n-1}&-2x^n),-y^{n+1}+\widehat{y}\rangle,
\end{aligned}
\end{equation}
where $x^{-1}:=x^0$.
Let us evaluate now the last term in relation \eqref{ineq-x-y-K}. For any $n \geq 0$ it holds
\begin{equation}\label{ev-K}
\begin{aligned}
\langle K(x^{n+1}&+x^{n-1}-2x^n),-y^{n+1}+\widehat{y}\rangle=\\
 \langle K(x^{n+1}-x^n),-y^{n+1}+\widehat{y}\rangle & +\langle K(x^n-x^{n-1}),y^n-\widehat{y}\rangle + \langle K(x^n-x^{n-1}),y^{n+1}-y^n\rangle \leq\\
-\langle K(x^{n+1}-x^n),y^{n+1}-\widehat{y}\rangle & +\langle K(x^n-x^{n-1}),y^n-\widehat{y}\rangle+\|K\|\|x^n-x^{n-1}\|\|y^{n+1}-y^n\| \leq\\
 -\langle K(x^{n+1}&-x^n),y^{n+1}-\widehat{y}\rangle +\langle K(x^n-x^{n-1}),y^n-\widehat{y}\rangle+\\
\frac{\sqrt{\sigma\tau}\|K\|}{2\tau} & \|x^n-x^{n-1}\|^2  +\frac{\sqrt{\sigma\tau}\|K\|}{2\sigma} \|y^{n+1}-y^n\|^2.
\end{aligned}
\end{equation}
From \eqref{ineq-x-y-K} and \eqref{ev-K} we obtain for any $n\geq 0$ the following estimation
$$\frac{1}{2\tau}\|x^{n+1}-\widehat{x}\|^2+\frac{1}{2\sigma}\|y^{n+1}-\widehat{y}\|^2\leq
\frac{1}{2\tau}\|x^n-\widehat{x}\|^2+\frac{1}{2\sigma}\|y^n-\widehat{y}\|^2-$$
$$\frac{1}{2\tau}\|x^{n+1}-x^n\|^2 -\frac{1}{2\sigma}\|y^{n+1}-y^n\|^2
-\langle K(x^{n+1}-x^n),y^{n+1}-\widehat{y}\rangle+\langle K(x^n-x^{n-1}),y^n-\widehat{y}\rangle+$$
$$\frac{\sqrt{\sigma\tau}\|K\|}{2\tau} \|x^n-x^{n-1}\|^2+\frac{\sqrt{\sigma\tau}\|K\|}{2\sigma}
\|y^{n+1}-y^n\|^2,$$
thus
\begin{equation}\label{ineq-x-y-K-de-sumat}
\begin{aligned}
\frac{\|x^{n+1}-\widehat{x}\|^2}{2\tau}+&\frac{\|y^{n+1}-\widehat{y}\|^2}{2\sigma} \leq\frac{\|x^n-\widehat{x}\|^2}{2\tau}+
\frac{\|y^n-\widehat{y}\|^2}{2\sigma}+\\
(-1+\sqrt{\sigma\tau}\|K\|) \frac{\|y^{n+1}-y^n\|^2}{2\sigma}&-\frac{\|x^{n+1}-x^n\|^2}{2\tau}+\sqrt{\sigma\tau}\|K\|\frac{\|x^n-x^{n-1}\|^2}{2\tau}\\
-\langle K(x^{n+1}-x^n),&y^{n+1}-\widehat{y}\rangle+\langle K(x^n-x^{n-1}),y^n-\widehat{y}\rangle.
\end{aligned}
\end{equation}
Let be an arbitrary $N\in\N$, $N\geq 2$. Summing up the inequalities in \eqref{ineq-x-y-K-de-sumat} from $n=0$ to $N-1$ we obtain
\begin{equation}\label{ineq-x-y-K-sumat1}
\begin{aligned}
\frac{\|x^N-\widehat{x}\|^2}{2\tau}+\frac{\|y^N-\widehat{y}\|^2}{2\sigma}&\leq\frac{\|x^0-\widehat{x}\|^2}{2\tau}+
\frac{\|y^0-\widehat{y}\|^2}{2\sigma}+\\
(-1+\sqrt{\sigma\tau}\|K\|) \sum_{n=1}^{N}\frac{\|y^n-y^{n-1}\|^2}{2\sigma}+(-1&+\sqrt{\sigma\tau}\|K\|) \sum_{n=1}^{N-1}\frac{\|x^n-x^{n-1}\|^2}{2\tau}-\\
\frac{\|x^N-x^{N-1}\|^2}{2\tau}&-\langle K(x^N-x^{N-1}),y^N-\widehat{y}\rangle.
\end{aligned}
\end{equation}
By combining \eqref{ineq-x-y-K-sumat1} with $$-\langle K(x^N-x^{N-1}),y^N-\widehat{y}\rangle\leq \frac{\|x^N-x^{N-1}\|^2}{2\tau}+\frac{\sigma\tau\|K\|^2}{2\sigma}\|y^N-\widehat{y}\|^2,$$ we get $$\frac{\|x^N-\widehat{x}\|^2}{2\tau}+\frac{\|y^N-\widehat{y}\|^2}{2\sigma}\leq\frac{\|x^0-\widehat{x}\|^2}{2\tau}+
\frac{\|y^0-\widehat{y}\|^2}{2\sigma}+$$$$(-1+\sqrt{\sigma\tau}\|K\|)
\sum_{n=1}^{N}\frac{\|y^n-y^{n-1}\|^2}{2\sigma}+(-1+\sqrt{\sigma\tau}\|K\|)
\sum_{n=1}^{N-1}\frac{\|x^n-x^{n-1}\|^2}{2\tau}+\frac{\sigma\tau\|K\|^2}{2\sigma}\|y^N-\widehat{y}\|^2$$
or, equivalently,
\begin{equation}\label{bound-seq-ser}
\begin{aligned}
\frac{\|x^N-\widehat{x}\|^2}{2\tau}&+(1-\sigma\tau\|K\|^2)\frac{\|y^N-\widehat{y}\|^2}{2\sigma}+\\
(1-\sqrt{\sigma\tau}\|K\|) \sum_{n=1}^{N}\frac{\|y^n-y^{n-1}\|^2}{2\sigma}&+(1-\sqrt{\sigma\tau}\|K\|)
\sum_{n=1}^{N-1}\frac{\|x^n-x^{n-1}\|^2}{2\tau}\leq\\
\frac{\|x^0-\widehat{x}\|^2}{2\tau}&+\frac{\|y^0-\widehat{y}\|^2}{2\sigma}.
\end{aligned}
\end{equation}
By taking into account that $\sigma\tau\|K\|^2<1$ \eqref{bound-seq-ser} yields \eqref{bound-seq}, hence $(x^n,y^n)_{n \geq 0}$ is bounded.

(ii) According to (i), $(x^n,y^n)_{n \geq 0}$ has a subsequence $(x^{n_k},y^{n_k})_{k \geq 0}$ which converges to an element $(x^*,y^*) \in X \times Y$ as $k\rightarrow+\infty$. From \eqref{gr-A} and \eqref{gr-B} and using that, due to the maximal monotonicity of $A$ and $B$,  $\gr A$ and $\gr B$ are closed sets, it follows that $(x^*,y^*)$ is a solution of the system of inclusions \eqref{syst-incl-2}. On the other hand, from \eqref{bound-seq-ser} we obtain that $\lim_{n\rightarrow+\infty}(x^n-x^{n-1})=\lim_{n\rightarrow+\infty}(y^n-y^{n-1})=0$.

Further, let be $k \geq 0$ and $N \in \N$, $N > n_k$. Summing up the inequalities in \eqref{ineq-x-y-K-de-sumat}, for $(\widehat{x},\widehat{y}):=(x^*,y^*)$, from $n=n_k$ to $N-1$ we obtain
$$\frac{\|x^N-x^*\|^2}{2\tau}+\frac{\|y^N-y^*\|^2}{2\sigma}+(1-\sqrt{\sigma\tau}\|K\|)
\sum_{n=n_k+1}^{N}\frac{\|y^n-y^{n-1}\|^2}{2\sigma}$$
$$-\frac{\|x^{n_k}-x^{n_k-1}\|^2}{2\tau}+(1-\sqrt{\sigma\tau}\|K\|)\sum_{n=n_k}^{N-1}\frac{\|x^n-x^{n-1}\|^2}{2\tau}+\frac{\|x^N-x^{N-1}\|^2}{2\tau}$$
$$+\langle K(x^N-x^{N-1}),y^N-y^*\rangle-\langle K(x^{n_k}-x^{n_k-1}),y^{n_k}-y^*\rangle$$
$$\leq \frac{\|x^{n_k}-x^*\|^2}{2\tau}+\frac{\|y^{n_k}-y^*\|^2}{2\sigma},$$
which yields
$$\frac{\|x^N-x^*\|^2}{2\tau}+\frac{\|y^N-y^*\|^2}{2\sigma} \leq \|K\|\|x^N-x^{N-1}\|\|y^N-y^*\| + $$
$$\frac{\|x^{n_k}-x^*\|^2}{2\tau}+\frac{\|y^{n_k}-y^*\|^2}{2\sigma} +  \frac{\|x^{n_k}-x^{n_k-1}\|^2}{2\tau} + \langle K(x^{n_k}-x^{n_k-1}),y^{n_k}-y^*\rangle.$$
Consequently, by using also the boundedness of $(x^n,y^n)_{n \geq 0}$, for any $k \geq 0$ it holds
$$\limsup_{N \rightarrow +\infty}\left(\frac{\|x^N-x^*\|^2}{2\tau}+\frac{\|y^N-y^*\|^2}{2\sigma} \right) \leq$$
$$\frac{\|x^{n_k}-x^*\|^2}{2\tau}+\frac{\|y^{n_k}-y^*\|^2}{2\sigma} +  \frac{\|x^{n_k}-x^{n_k-1}\|^2}{2\tau} + \langle K(x^{n_k}-x^{n_k-1}),y^{n_k}-y^*\rangle.$$
We finally let $k$ converge to $+\infty$, which yields
$$\limsup_{N \rightarrow +\infty}\left(\frac{\|x^N-x^*\|^2}{2\tau}+\frac{\|y^N-y^*\|^2}{2\sigma} \right) = 0$$
and, further, $\lim_{N\rightarrow+\infty}x^N=x^*$ and $\lim_{N\rightarrow+\infty}y^N=y^*$.
\end{proof}

\begin{remark}\label{arhur}
The characterization of the solution of the system of inclusions \eqref{syst-incl-2} given in \eqref{motiv} motivates the following iterative scheme

\noindent\begin{tabular}{rl}
\verb"Initialization": & \verb"Choose" $\sigma,\tau>0$ \verb"and" $(x^0,y^0)\in X\times Y$.\\
\verb"For" $n\geq 0$ \verb"set": & $y^{n+1}:=J_{\sigma B^{-1}}(y^n+\sigma Kx^n)$\\
& $x^{n+1}:=J_{\tau A}(x^n-\tau K^*y^{n+1})$
\end{tabular}

\noindent as well, which is nothing else than an Arrow-Hurwicz-Uzawa-type algorithm (see \cite{ahu}) designed for the inclusion problem \eqref{ecu1}.
\end{remark}

We close this section by discussing another modality of investigating the system of inclusions \eqref{syst-incl-2} by employing some ideas considered in  \cite{br-combettes, combettes-pesquet}. To this end we define the operators
$M:X\times Y\rightrightarrows X\times Y$, $M(x,y)=(Ax,B^{-1}y)$, and $S:X\times Y\rightarrow X\times Y$, $S(x,y)=(K^*y,-Kx)$. The operator $M$ is maximally monotone, since $A$ and $B$ are maximally monotone, while $S$ is maximally monotone, since it is a skew linear operator. Then $(\widehat{x},\widehat{y}) \in X \times Y$ is a solution of the system of inclusions \eqref{syst-incl-2} if and only if it solves the inclusion problem
\begin{equation}\label{M-S}
\mbox{find} \ (x,y) \in X \times Y \ \mbox{such that} \ (0,0)\in S(x,y) + M(x,y).
\end{equation}

Applying Algorithm \ref{alg-2-op} to the  problem \eqref{M-S} with starting point $(x^{0},y^{0}, u^{0},v^{0}) \in X \times Y \times X \times Y$, $(\ol {x}^0,\ol {y}^0)=(x^{0},y^{0})$ and $\sigma, \tau > 0$ gives rise for any $n \geq 0$ to the following iterations:
$$\begin{array}{lll}
(u^{n+1},v^{n+1}):=J_{\sigma M^{-1}}\big[(u^{n},v^{n})+\sigma (\ol{x}^n,\ol{y}^n)\big]\\
(x^{n+1},y^{n+1}):=J_{\tau S}\big[(x^{n},y^{n})-\tau (u^{n+1},v^{n+1})\big]\\
(\ol{x}^{n+1},\ol{y}^{n+1}):=2(x^{n+1},y^{n+1})-(x^{n},y^{n}).
\end{array}$$

Since
$$J_{\sigma M^{-1}}=J_{\sigma A^{-1}}\times J_{\sigma B}$$
and (see \cite[Proposition 2.7]{br-combettes})
$$J_{\tau S}(x,y)=\big((\id\nolimits_X+\tau^2 K^*K)^{-1}(x-\tau K^*y),(\id\nolimits_Y+\tau^2 KK^*)^{-1}(y+\tau Kx)\big) \ \forall (x,y) \in X \times Y,$$
this yields the following algorithm:

\begin{algorithm}\label{alg-2-op-combettes}$ $

\noindent\begin{tabular}{rl}
\verb"Initialization": & \verb"Choose" $\sigma,\tau>0$ \verb"such that" $\sigma\tau <1$ \verb"and" $(x^{0},y^{0}), (u^{0},v^{0}) \in X \times Y$.\\
& \verb"Set" $(\ol {x}^0,\ol {y}^0):=(x^{0},y^{0})$.\\
\verb"For" $n\geq 0$ \verb"set": & $u^{n+1}:=J_{\sigma A^{-1}}(u^{n}+\sigma\ol{x}^n)$\\
                                 & $v^{n+1}:=J_{\sigma B}(v^{n}+\sigma\ol{y}^n)$\\
                                 & $x^{n+1}:=(\id_X+\tau^2 K^*K)^{-1}\big[x^{n}-\tau u^{n+1}-\tau K^*(y^{n}-\tau v^{n+1})\big]$\\
                                 & $y^{n+1}:=(\id_Y+\tau^2 KK^*)^{-1}\big[y^{n}-\tau v^{n+1}+\tau K(x^{n}-\tau u^{n+1})\big]$\\
                                 & $\ol{x}^{n+1}:=2x^{n+1}-x^{n}$\\
                                 & $\ol{y}^{n+1}:=2y^{n+1}-y^{n}$
\end{tabular}

\end{algorithm}

The following convergence statement is a consequence of Theorem \ref{th-2-op}.

\begin{theorem} \label{th-2-op-comb} Assume that $X$ and $Y$ are finite-dimensional spaces and that the system of inclusions \eqref{syst-incl-2} is solvable. Then the sequence $(x^n,y^n)_{n \geq 0}$ generated in Algorithm \ref{alg-2-op-combettes} converges to $(x^*,y^*)$, a solution of the system of inclusions \eqref{syst-incl-2}, which yields that $x^*$ is a solution of the primal inclusion problem \eqref{ecu1} and $y^*$ is a solution of the dual inclusion problem \eqref{ecu2}.
\end{theorem}

\begin{remark}\label{noteasy} As we have already mentioned, the system of inclusions \eqref{syst-incl-2} is solvable if and only if the primal inclusion problem \eqref{ecu1} is solvable,  which is further equivalent to
solvability of the dual inclusion problem \eqref{ecu2}. Let us also notice that from the point of view of the numerical implementation Algorithm \ref{alg-2-op-combettes}  has the drawback to ask for the calculation of the inverses of $\id_X+\tau^2 K^*K$ and $\id_Y+\tau^2 KK^*$. This task can be in general very hard, but it becomes very simple when $K$ is, for instance, \textit{orthogonal}, like it happens for the linear transformations to which  \textit{orthogonal wavelets} give rise and which play an important role in signal processing.
\end{remark}

\section{Zeros of sums of compositions of monotone operators with linear continuous operators}\label{sec3}

In this section we provide via the primal-dual scheme Algorithm \ref{alg-2-op} an algorithm for solving the inclusion problem
\begin{equation}\label{sum-k-prim}
\mbox{find} \ x \in X \ \mbox{such that} \ 0\in\sum_{i=1}^{k}\omega_iK_i^*B_iK_ix,
\end{equation}
where $X$ and $Y_i$ are real Hilbert spaces, $B_i:Y_i\rightrightarrows Y_i$ are maximally monotone operators, $K_i:X\rightarrow Y_i$ are linear and continuous operators for $i=1,...,k$ and $\omega_i\in(0,1], i=1,...,k,$ are real numbers fulfilling $\sum_{i=1}^{k}\omega_i=1$. The dual inclusion problem of \eqref{sum-k-prim} reads
\begin{equation}\label{sum-k-dual}
\mbox{ find } y=(y_1,...,y_k) \in Y_1\times...\times Y_k \ \mbox{such that} \ \sum_{i=1}^{k}\omega_iK_i^*y_i=0 \ \mbox{and} \ \bigcap_{i=1}^k (B_iK_i)^{-1}(y_i) \neq \emptyset.
\end{equation}
Following the product space approach from \cite{br-combettes} (see also \cite{bauschke-book}) we show that this primal-dual pair can be reduced to a primal-dual pair of inclusion problems of the form \eqref{ecu1}-\eqref{ecu2}.

Consider the real Hilbert space $H:=X^k$  endowed with the inner product $\langle x,u\rangle_H=\sum_{i=1}^{k}\omega_i\langle x_i,u_i\rangle_X$ for $x=(x_i)_{1\leq i\leq k}, u=(u_i)_{1\leq i\leq k} \in H$, where $\langle \cdot,\cdot \rangle_X$ denotes the inner product on $X$. Further, let $Y:=Y_1\times...\times Y_k$ be the real Hilbert space endowed with the inner product $\langle y,z\rangle_Y:=\sum_{i=1}^{k}\omega_i\langle y_i,z_i\rangle_{Y_i}$
for $y=(y_i)_{1\leq i\leq k}, z=(z_i)_{1\leq i\leq k} \in Y$, where $\langle \cdot,\cdot \rangle_{Y_i}$ denotes the inner product on $Y_i, i=1,...,k$.
We define $A : H \rightrightarrows H$, $A:=N_V$, where $V=\{(x,...,x)\in H:x\in X\}$, $B:Y\rightrightarrows Y$, $B(y_1,...,y_k)=(B_1y_1,...,B_ky_k)$, and
$K:H\rightarrow Y$, $K(x_1,...,x_k)=(K_1x_1,...,K_kx_k)$. Obviously, the adjoint operator of $K$ is $K^*:Y\rightarrow H$, $K^*(y_1,...,y_k)=(K_1^*y_1,...,K_k^*y_k)$, for $(y_1,...,y_k) \in Y$. Further, let be $j:X\rightarrow H$, $j(x)=(x,...,x)$.

The operators $A$ and $B$ are maximally monotone and
$$x \mbox{ solves }\eqref{sum-k-prim} \mbox{ if and only if } (0,...,0) \in A(j(x))+K^*BK(j(x)),$$ while
$$y=(y_1,...,y_k) \mbox{ solves }\eqref{sum-k-dual} \mbox{ if and only if } (0,...,0)\in B^{-1}y-KA^{-1}(-K^*)y.$$

Applying Algorithm \ref{alg-2-op} to the inclusion problem
\begin{equation}\label{ecu3}
\mbox{find} \ (x_1,...,x_k) \in H \ \mbox{such that} \  0\in A(x_1,...,x_k) + K^*BK(x_1,...,x_k)
\end{equation}
with starting point $(x^{0},...,x^0,y_1^{0},...,y_k^0) \in \underbrace{X \times...\times X}_k \times Y_1 \times...\times Y_k$, constants $\sigma, \tau > 0$  and $(\ol {x}_1^0,...,\ol x_k^0)$ $:= (x^{0},...,x^0)$ yields for any $n \geq 0$ the following iterations:
$$\begin{array}{lll}
(y^{n+1}_i)_{1\leq i\leq k}:=J_{\sigma B^{-1}}\Big((y^{n}_i)_{1\leq i\leq k}+\sigma K (\ol{x}^{n}_i)_{1\leq i\leq k}\Big)\\
({x}^{n+1}_i)_{1\leq i\leq k}:=J_{\tau A}\Big(({x}^{n}_i)_{1\leq i\leq k}-\tau K^*(y^{n+1}_i)_{1\leq i\leq k}\Big)\\
(\ol{x}^{n+1}_i)_{1\leq i\leq k} :=2({x}^{n+1}_i)_{1\leq i\leq k} - ({x}^{n}_i)_{1\leq i\leq k}.
\end{array}$$
According to \cite{br-combettes}, for the occurring resolvents we have that $J_{\tau A}(u_1,...,u_k)=j(\sum_{i=1}^{k}\omega_iu_i)$ for $(u_1,...,u_k) \in H$ and $J_{\sigma B^{-1}}(z_1,...,z_k)=(J_{\sigma B_1^{-1}}z_1,...,J_{\sigma B_k^{-1}}z_k )$ for $(z_1,...,z_k) \in Y$. This means that for any $n \geq 1$ it holds $x^n_1 =...=x^n_k$ and $\ol x^{n+1}_1=...=\ol x^{n+1}_k$, which shows that there is no loss in the generality of the algorithm when assuming that the first $k$ components of the starting point coincide. Notice that a solution $(\widehat x_1,...,\widehat x_k)$ of \eqref{ecu3} must belong to $\dom A$, thus $\widehat x_1 = ...=\widehat x_k$. We obtain the following algorithm:

\begin{algorithm}\label{alg-k-op}$ $

\noindent\begin{tabular}{rl}
\verb"Initialization": & \verb"Choose" $\sigma,\tau>0$ \verb"such that" $\sigma\tau\sum_{i=1}^{k}\|K_i\|^2<1$ \verb"and"\\
                       & $(x^{0},y_1^{0},...,y_k^0) \in X \times Y_1 \times...\times Y_k$. \verb"Set" $\ol {x}^0:=x^0$.\\
\verb"For" $n\geq 0$ \verb"set": & $y^{n+1}_i:=J_{\sigma B_i^{-1}}(y^{n}_i+\sigma K_i\ol{x}^{n}), i=1,...,k$\\
                                 & $x^{n+1}:=x^n - \tau \sum_{i=1}^{k}\omega_i K_i^*y^{n+1}_i$\\
                                 & $\ol{x}^{n+1}:=2x^{n+1}-x^{n}$
\end{tabular}
\end{algorithm}

The convergence of Algorithm \ref{alg-k-op} is stated by the following result which is a consequence of Theorem \ref{th-2-op}.

\begin{theorem} \label{th-k-op} Assume that $X$ and $Y_i, i=1,...,k,$ are finite-dimensional spaces and \eqref{sum-k-prim} is solvable. Then \eqref{sum-k-dual} is also solvable and the sequences $(x^n)_{n \geq 0}$ and $(y^n_1,...,y^n_k)_{n \geq 0}$ generated in Algorithm \ref{alg-k-op} converge to a solution of the primal inclusion problem \eqref{sum-k-prim} and to a solution of the dual inclusion problem \eqref{sum-k-dual}, respectively.\end{theorem}

\begin{remark}\label{remsum}
Since $\|K\|^2 \leq \sum_{i=1}^{k}\|K_i\|^2$, the inequality $\sigma\tau\sum_{i=1}^{k}\|K_i\|^2<1$ in  Algorithm \ref{alg-k-op} is considered in order to ensure that $\sigma\tau\|K\|^2<1$.
\end{remark}

When particularizing the above framework to the case when  $Y_i=X$ and $K_i=\id_X$ for $i=1,..,k$, the primal-dual pair of inclusion problems \eqref{sum-k-prim}-\eqref{sum-k-dual} become
\begin{equation}\label{sum-k-prim-id}
\mbox{ find }x\in X \mbox{ such that } 0\in\sum_{i=1}^{k}\omega_iB_ix
\end{equation}
and
\begin{equation}\label{sum-k-dual-id}
\mbox{ find } y=(y_1,...,y_k) \in X \times...\times X \ \mbox{such that} \ \sum_{i=1}^{k}\omega_iy_i=0 \ \mbox{and} \ \bigcap_{i=1}^k B_i^{-1}(y_i) \neq \emptyset,\end{equation}
respectively. In this situation $H=Y$, $K=\id_H$, $\|K\|=1$ and
$$x \mbox{ solves } \eqref{sum-k-prim-id} \mbox{ if and only if } (0,...,0) \in A(j(x)) + B(j(x)),$$
while
$$y=(y_1,...,y_k) \mbox{ solves }\eqref{sum-k-dual-id} \mbox{ if and only if }  (0,...,0)\in B^{-1}(y)-A^{-1}(-y).$$

Algorithm \ref{alg-k-op} yields in this particular case the following iterative scheme:

\begin{algorithm}\label{alg-k-op-id1}$ $

\noindent\begin{tabular}{rl}
\verb"Initialization": & \verb"Choose" $\sigma,\tau>0$ \verb"such that" $\sigma\tau <1$ \verb"and"\\
                       & $(x^{0},y_1^{0},...,y_k^0) \in \underbrace{X \times...\times X}_{k+1}$. \verb"Set" $\ol {x}^0:=x^0$.\\
\verb"For" $n\geq 0$ \verb"set": & $y^{n+1}_i:=J_{\sigma B_i^{-1}}(y^{n}_i+\sigma\ol{x}^{n}), i=1,...,k$\\
                                 & $x^{n+1}:=x^n - \tau \sum_{i=1}^{k}\omega_iy^{n+1}_i$\\
                                 & $\ol{x}^{n+1}:=2x^{n+1}-x^{n}$
\end{tabular}

\end{algorithm}

The convergence of Algorithm \ref{alg-k-op-id1} follows via Theorem \ref{th-k-op}.

\begin{theorem} \label{th-k-op-id1} Assume that $X$ is a finite-dimensional space and \eqref{sum-k-prim-id} is solvable. Then \eqref{sum-k-dual-id} is also solvable and the sequences $(x^n)_{n \geq 0}$ and $(y^n_1,...,y^n_k)_{n \geq 0}$ generated in Algorithm \ref{alg-k-op-id1} converge to a solution of the primal inclusion problem \eqref{sum-k-prim-id} and to a solution of the dual inclusion problem \eqref{sum-k-dual-id}, respectively.\end{theorem}

In the last part of this section we provide a second algorithm which solves \eqref{sum-k-prim-id} and \eqref{sum-k-dual-id} which starts from the premise that by changing the roles of $A$ and $B$ one has
$$x \mbox{ solves } \eqref{sum-k-prim-id} \mbox{ if and only if } (0,...,0) \in B(j(x)) + A(j(x)),$$
while
$$y=(y_1,...,y_k) \mbox{ solves }\eqref{sum-k-dual-id} \mbox{ if and only if }  (0,...,0)\in A^{-1}(-y)-B^{-1}(y).$$
Applying Algorithm \ref{alg-2-op} to the inclusion problem
\begin{equation*}
\mbox{find} \ (x_1,...,x_k) \in H \ \mbox{such that} \  0\in B(x_1,...,x_k) + A(x_1,...,x_k)
\end{equation*}
with starting point $(x_1^{0},...,x_k^0, y_1^{0},...,y_k^0) \in \underbrace{X \times...\times X}_{2k}$, constants $\sigma, \tau > 0$ and $\!(\ol {x}_1^0,...,\ol x_k^0)\!\! :=$ $(x_1^{0},...,x_k^0)$ yields for any $n \geq 0$ the following iterations:
$$\begin{array}{lll}
(y^{n+1}_i)_{1\leq i\leq k}=J_{\sigma A^{-1}}\Big((y^{n}_i)_{1\leq i\leq k}+\sigma (\ol{x}^{n}_i)_{1\leq i\leq k}\Big)\\
(x^{n+1}_i)_{1\leq i\leq k}=J_{\tau B}\Big((x^{n}_i)_{1\leq i\leq k}-\tau (y^{n+1}_i)_{1\leq i\leq k}\Big)\\
(\ol{x}^{n+1}_i)_{1\leq i\leq k}=2(x^{n+1}_i)_{1\leq i\leq k}-(x^{n}_i)_{1\leq i\leq k}.
\end{array}$$
Noticing that $J_{\sigma A^{-1}}=J_{\sigma N_V^{-1}}=\id_H-\sigma J_{\sigma^{-1}N_V}\circ\sigma^{-1}\id_H$ (cf. \cite[Proposition 23.18]{bauschke-book}) and $J_{\sigma^{-1}N_V}(u_1,...,u_k)=J_{N_V}(u_1,...,u_k)=j(\sum_{i=1}^{k}\omega_iu_i)$ for $(u_1,...,u_k) \in H$ (cf. \cite[relation (3.27)]{br-combettes}) and by making for any $n \geq 0$ the change of variables $y^n_i:=-y^n_i$ for $i=1,...,k$, we obtain the following iterative scheme:

\begin{algorithm}\label{alg-k-op-id}$ $

\noindent\begin{tabular}{rl}
\verb"Initialization": & \verb"Choose" $\sigma,\tau>0$ \verb"such that" $\sigma\tau < 1$ \verb"and"\\
                       & $(x^{0}_1,...,x^0_k,y_1^{0},...,y_k^0) \in \underbrace{X \times...\times X}_{2k}$. \verb"Set" $(\ol {x}_1^0,...,\ol x_k^0) := (x_1^{0},...,x_k^0)$.\\
\verb"For" $n\geq 0$ \verb"set": & $y^{n+1}_i:=y^{n}_i-\sigma\ol{x}^{n}_i+\sum_{j=1}^{k}\omega_j y^{n}_j+\sigma \sum_{j=1}^{k}\omega_j \ol{x}^{n}_j, i=1,...,k$\\
                                 & $x^{n+1}_i:=J_{\tau B_i}(x^{n}_i+\tau y^{n+1}_i), i=1,...,k$\\
                                 & $\ol{x}^{n+1}_i:=2x^{n+1}_i-x^{n}_i, i=1,...,k$
\end{tabular}

\end{algorithm}

\begin{theorem} \label{th-k-op-id} Assume that $X$ is finite dimensional and \eqref{sum-k-prim-id} is solvable. Then \eqref{sum-k-dual-id} is also solvable and for all $i=1,...,k$ the sequence $(x^{n}_i)_{n \geq 0}$ generated in Algorithm \ref{alg-k-op-id} converges to a solution of \eqref{sum-k-prim-id} and the sequence $(y^{n}_1,...,y^n_k)_{n \geq 0}$ generated by the same algorithm converges to a solution of \eqref{sum-k-dual-id}.
\end{theorem}

\section{Solving convex optimization problems via the primal-dual algorithm}\label{sec4}

The aim of this section is to employ the iterative methods investigated above for solving several classes of unconstrained convex optimization problems. To this end we consider first for the real Hilbert spaces $X$ and $Y$ the proper, convex and lower semicontinuous functions $f:X\rightarrow \overline{\R}$ and $g:Y \rightarrow \overline{\R}$ and the linear and continuous operator $K:X\rightarrow Y$ the optimization problem
\begin{equation}\label{opt-pr-2}
\inf_{x\in X}\{f(x)+g(Kx)\}
\end{equation}
along with its \textit{Fenchel dual} problem (see \cite{bauschke-book, b-hab, EkTem, Zal-carte})
\begin{equation}\label{opt-dual-2}
\sup_{y\in Y}\{-f^*(-K^*y)-g^*(y)\}.
\end{equation}
For this primal-dual pair \textit{weak duality} always holds, i.e., the optimal objective value of the primal problem is greater than or equal to the optimal objective value of the dual problem. In order to guarantee \textit{strong duality}, i.e., the situation when the optimal objective values of the two problems coincide and the dual problem has an optimal solution one needs to ask for the fulfillment of a so-called \textit{qualification condition}. Some of the most popular interiority-type qualification conditions are (see, for instance, \cite{b-hab, Zal-carte, bot-csetnek, simons, bauschke-book, EkTem}):
\begin{center}
\begin{tabular}{r|l}
$(QC_1)$ \ & \  $\exists x'\in \dom f \cap K^{-1}(\dom g)$ such that $g$ is continuous at $Kx'$,
\end{tabular}
\end{center}
\begin{center}
\begin{tabular}{r|l}
$(QC_2)$ \ & \  $0\in\inte(\dom g-K(\dom f))$
\end{tabular}
\end{center}
and
\begin{center}
\begin{tabular}{r|l}
$(QC_3)$ \ & \ $0\in\sqri(\dom g-K(\dom f))$.
\end{tabular}
\end{center}
We notice that $(QC_1)\Rightarrow(QC_2)\Rightarrow(QC_3)$, these implications being in general strict, and refer the reader to the works cited above and the references therein for other qualification conditions in convex optimization.

Algorithm \ref{alg-2-op} written for $A:=\partial f$ and $B:=\partial g$ yields the following iterative scheme:

\begin{algorithm}\label{alg-2-functions}$ $

\noindent\begin{tabular}{rl}
\verb"Initialization": & \verb"Choose" $\sigma,\tau>0$ \verb"such that" $\sigma\tau\|K\|^2<1$ \verb"and" $(x^0,y^0)\in X\times Y$.\\
& \verb"Set" $\ol{x}^0:=x^0$.\\
\verb"For" $n\geq 0$ \verb"set": & $y^{n+1}:=\prox_{\sigma g^*}(y^n+\sigma K\ol{x}^n)$\\
& $x^{n+1}:=\prox_{\tau f}(x^n-\tau K^*y^{n+1})$\\
& $\ol{x}^{n+1}:=2x^{n+1}-x^n$
\end{tabular}

\end{algorithm}

We have the following convergence result.

\begin{theorem}\label{th-2-functions} Assume that the primal problem \eqref{opt-pr-2} has an optimal solution $\widehat{x}$ and one of the qualification conditions $(QC_i)$, $i = 1,2,3$, is fulfilled. Let $(x^n,\ol{x}^n,y^n)_{n \geq 0}$ be the sequence generated by Algorithm \ref{alg-2-functions}. The following statements are true:

(i) There exists $\widehat{y} \in Y$, an optimal solution of the dual problem \eqref{opt-dual-2}, the optimal objective values of the two optimization problems coincide and $(\widehat{x},\widehat{y})$ is a solution of the system of inclusions
\begin{equation}\label{syst-incl-2-f}
Kx\in \partial g^*(y) \ \mbox{and} \ -K^*y\in \partial f(x).
\end{equation}

(ii) For any $n \geq 0$ it holds
\begin{equation}\label{bound-seq-f} \frac{\|x^n-\widehat{x}\|^2}{2\tau}+(1-\sigma\tau\|K\|^2)\frac{\|y^n-\widehat{y}\|^2}{2\sigma}\leq
\frac{\|x^0-\widehat{x}\|^2}{2\tau}+\frac{\|y^0-\widehat{y}\|^2}{2\sigma},
\end{equation}
thus the sequence $(x^n,y^n)_{n \geq 0}$ is bounded.

(iii) If $X$ and $Y$ are finite-dimensional, then $(x^n)_{n \geq 0}$ converges to an optimal solution of \eqref{opt-pr-2} and $(y^n)_{n \geq 0}$ converges to an optimal solution of \eqref{opt-dual-2}.
\end{theorem}

\begin{remark}\label{remfunctions}
(i) Statement (i) of the above theorem is well-known in the literature,  \eqref{syst-incl-2-f} being nothing else than the system of optimality conditions for the primal-dual pair \eqref{opt-pr-2}-\eqref{opt-dual-2} (see, for instance, \cite{b-hab, EkTem, Zal-carte}), while the other two statements follow  from Theorem \ref{th-2-op}.

(ii) The existence of optimal solutions of the primal problem \eqref{opt-pr-2} is guaranteed if, for instance, $f$ is coercive and $g$ is bounded below. Indeed, under these circumstances, the objective function of \eqref{opt-pr-2} is coercive and the statement follows via \cite[Theorem 2.5.1(ii)]{Zal-carte} (see, also, \cite[Proposition 15.7]{bauschke-book}). On the other hand, when $f$ is strongly convex, then $f + g \circ K$ is strongly convex, too, thus \eqref{opt-pr-2} has an unique optimal solution (cf. \cite[Corollary 11.16]{bauschke-book}).

(iii) We rediscovered above the iterative scheme and the convergence statement from \cite{ch-pck} as a particular instance of the general results furnished in Section \ref{sec2}.
\end{remark}

For $X$ and $Y_i$  real Hilbert spaces, $g_i:Y_i\rightarrow \overline \R$ proper, convex and lower semicontinuous functions, $K_i:X\rightarrow Y_i$ linear and continuous operators and $\omega_i\in(0,1], i=1,...,k,$ real numbers fulfilling $\sum_{i=1}^{k}\omega_i=1$ consider the optimization problem
\begin{equation}\label{sum-k-prim-f}
\inf_{x\in X}\sum_{i=1}^{k}\omega_ig_i(K_ix)
\end{equation}
and its \textit{Fenchel-type dual} problem
\begin{equation}\label{sum-k-dual-f}
\sup_{\substack{y_i\in Y_i, i=1,...,k\\ \sum_{i=1}^{k}\omega_iK_i^*y_i=0}} \sum_{i=1}^{k}-\omega_ig_i^*(y_i).
\end{equation}
For the primal-dual pair \eqref{sum-k-prim-f}-\eqref{sum-k-dual-f} strong duality holds whenever one of the following qualification conditions is fulfilled (see, for instance, \cite{b-hab, Zal-carte, br-combettes}):
\begin{center}
\begin{tabular}{r|l}
$(QC^\Sigma_1)$ \ & \  $\exists x' \in \bigcap_{i=1}^k K_i^{-1}(\dom g_i)$ such that $g_i$ is continuous at $K_ix', i=1,...,k$,
\end{tabular}
\end{center}
\begin{center}
\begin{tabular}{r|l}
$(QC^\Sigma_2)$ \ & \  $0 \in \inte \Big(\prod_{i=1}^{k}\dom g_i-\{(K_1x,...,K_kx):x\in X\}\Big)$
\end{tabular}
\end{center}
and
\begin{center}
\begin{tabular}{r|l}
$(QC^\Sigma_3)$ \ & \ $0\in\sqri\Big(\prod_{i=1}^{k}\dom g_i-\{(K_1x,...,K_kx):x\in X\}\Big)$.
\end{tabular}
\end{center}
Again, $(QC^\Sigma_1)\Rightarrow(QC^\Sigma_2)\Rightarrow(QC^\Sigma_3)$, the implications being in general strict. By taking $B_i : =\partial g_i$, $i=1,...,k$, Algorithm \ref{alg-k-op} yields the following iterative scheme:

\begin{algorithm}\label{alg-k-f}$ $

\noindent\begin{tabular}{rl}
\verb"Initialization": & \verb"Choose" $\sigma,\tau>0$ \verb"such that" $\sigma\tau\sum_{i=1}^{k}\|K_i\|^2<1$ \verb"and"\\
                       & $(x^{0},y_1^{0},...,y_k^0) \in X \times Y_1 \times...\times Y_k$. \verb"Set" $\ol {x}^0:=x^0$.\\
\verb"For" $n\geq 0$ \verb"set": & $y^{n+1}_i:=\prox_{\sigma g_i^*}(y^{n}_i+\sigma K_i\ol{x}^{n}), i=1,...,k$\\
                                 & $x^{n+1}:=x^n - \tau \sum_{i=1}^{k}\omega_i K_i^*y^{n+1}_i$\\
                                 & $\ol{x}^{n+1}:=2x^{n+1}-x^{n}$
\end{tabular}
\end{algorithm}

The convergence of Algorithm \ref{alg-k-f} is stated by the following result which is a consequence of Theorem \ref{th-k-op}.

\begin{theorem}\label{th-k-functions} Assume that the primal problem \eqref{sum-k-prim-f} has an optimal solution $\widehat{x}$ and one of the qualification conditions $(QC^\Sigma_i)$, $i = 1,2,3$, is fulfilled.  The following statements are true:

(i) There exists $(\widehat{y}_1,...,\widehat y_k) \in Y_1 \times ... \times Y_k$, an optimal solution of the dual problem \eqref{sum-k-dual-f}, the optimal objective values of the two optimization problems coincide and $(\widehat{x},\widehat{y}_1,...,\widehat y_k)$ is a solution of the system of inclusions
\begin{equation}\label{syst-incl-k-f}
K_ix\in \partial g_i^*(y_i), i=1,...,k, \ \mbox{and} \ \sum_{i=1}^{k}\omega_iK_i^*y_i=0.
\end{equation}

(ii) If $X$ and $Y$ are finite-dimensional, then the sequences $(x^n)_{n \geq 0}$ and $(y^n_1,...,y^n_k)_{n \geq 0}$ generated in Algorithm \ref{alg-k-f} converge to an optimal solution of \eqref{sum-k-prim-f} and \eqref{sum-k-dual-f}, respectively.
\end{theorem}

Considering, finally, the particular case when $Y_i=X$ and $K_i=\id_X$, $i=1,...,k$, the problems \eqref{sum-k-prim-f}
and \eqref{sum-k-dual-f} become
\begin{equation}\label{sum-k-prim-f-id} \inf_{x\in X}\sum_{i=1}^{k}\omega_ig_i(x) \end{equation}
and, respectively,
\begin{equation}\label{sum-k-dual-f-id}
\sup_{\substack{y_i\in X, i=1,...,k\\
\sum_{i=1}^{k}\omega_iy_i=0}}\sum_{i=1}^{k}-\omega_ig_i^*(y_i).
\end{equation}
The qualification conditions $(QC^\Sigma_i)$, $i = 1,2,3$, looks in this case like:
\begin{center}
\begin{tabular}{r|l}
$(QC^{id}_1)$ \ & \  $\exists x' \in \bigcap_{i=1}^k \dom g_i$ such that $g_i$ is continuous at $x', i=1,...,k$,
\end{tabular}
\end{center}
\begin{center}
\begin{tabular}{r|l}
$(QC^{id}_2)$ \ & \  $0 \in \inte \Big(\prod_{i=1}^{k}\dom g_i-\{(x,...,x):x\in X\}\Big)$
\end{tabular}
\end{center}
and, respectively,
\begin{center}
\begin{tabular}{r|l}
$(QC^{id}_3)$ \ & \ $0\in\sqri\Big(\prod_{i=1}^{k}\dom g_i-\{(x,...,x):x\in X\}\Big)$.
\end{tabular}
\end{center}
By particularizing Algorithm \ref{alg-k-f} we obtain:

\begin{algorithm}\label{alg-k-f-id1}$ $

\noindent\begin{tabular}{rl}
\verb"Initialization": & \verb"Choose" $\sigma,\tau>0$ \verb"such that" $\sigma\tau < 1$ \verb"and"\\
                       & $(x^{0},y_1^{0},...,y_k^0) \in \underbrace{X \times ...\times X}_{k+1}$. \verb"Set" $\ol {x}^0:=x^0$.\\
\verb"For" $n\geq 0$ \verb"set": & $y^{n+1}_i:=\prox_{\sigma g_i^*}(y^{n}_i+\sigma \ol{x}^{n}), i=1,...,k$\\
                                 & $x^{n+1}:=x^n - \tau \sum_{i=1}^{k}\omega_iy^{n+1}_i$\\
                                 & $\ol{x}^{n+1}:=2x^{n+1}-x^{n}$
\end{tabular}
\end{algorithm}
while Algorithm \ref{alg-k-op-id} gives rise to the following iterative scheme:

\begin{algorithm}\label{alg-k-f-id}$ $

\noindent\begin{tabular}{rl}
\verb"Initialization": & \verb"Choose" $\sigma,\tau>0$ \verb"such that" $\sigma\tau < 1$ \verb"and"\\
                       & $(x^{0}_1,...,x^0_k,y_1^{0},...,y_k^0) \in \underbrace{X \times...\times X}_{2k}$. \verb"Set" $(\ol {x}_1^0,...,\ol x_k^0) := (x_1^{0},...,x_k^0)$.\\
\verb"For" $n\geq 0$ \verb"set": & $y^{n+1}_i:=y^{n}_i-\sigma\ol{x}^{n}_i+\sum_{j=1}^{k}\omega_j y^{n}_j+\sigma \sum_{j=1}^{k}\omega_j \ol{x}^{n}_j, i=1,...,k$\\
                                 & $x^{n+1}_i:=\prox_{\tau g_i}(x^{n}_i+\tau y^{n+1}_i), i=1,...,k$\\
                                 & $\ol{x}^{n+1}_i:=2x^{n+1}_i-x^{n}_i, i=1,...,k$
\end{tabular}

\end{algorithm}
We have the following convergence theorem.

\begin{theorem}\label{th-k-f-id} Assume that the primal problem \eqref{sum-k-prim-f-id} has an optimal solution $\widehat{x}$ and one of the qualification conditions $(QC^{id}_i)$, $i = 1,2,3$, is fulfilled.  The following statements are true:

(i) There exists $(\widehat{y}_1,...,\widehat y_k) \in X \times ... \times X$, an optimal solution of the dual problem \eqref{sum-k-dual-f-id}, the optimal objective values of the two optimization problems coincide and $(\widehat{x},\widehat{y}_1,...,\widehat y_k)$ is a solution of the system of inclusions
\begin{equation}\label{syst-incl-k-f-id}
x\in \partial g_i^*(y_i), i=1,...,k, \ \mbox{and} \ \sum_{i=1}^{k}\omega_iy_i=0.
\end{equation}

(ii) If $X$ is finite-dimensional, then the sequences $(x^n)_{n \geq 0}$ and $(y^n_1,...,y^n_k)_{n \geq 0}$ generated in Algorithm \ref{alg-k-f-id1} converge to an optimal solution of  \eqref{sum-k-prim-f-id} and \eqref{sum-k-dual-f-id}, respectively.

(iii) If $X$ is finite-dimensional, then for all $i=1,...,k$ the sequence $(x^{n}_i)_{n \geq 0}$ generated in Algorithm \ref{alg-k-f-id} converges to an optimal solution of \eqref{sum-k-prim-f-id} and the sequence $(y^{n}_1,...y^n_k)_{n \geq 0}$ generated by the same algorithm converges to an optimal solution of \eqref{sum-k-dual-f-id}.
\end{theorem}

\begin{remark}\label{reg-cond-k-f-id} One can notice that Theorem \ref{th-k-f-id} remains valid even under a weaker condition than in $(QC^{id}_1)$, namely by assuming that there exists $x' \in \cap_{i=1}^{k}\dom g_i$ such that $k-1$  of the functions $g_i, i=1,...,k,$ are continuous at $x'$ (see \cite[Remark 2.5]{b-hab}).
\end{remark}

\section{Numerical experiments}\label{sec5}

In this section we present  numerical experiments involving the primal-dual algorithm and some of its variants when solving some nondifferentiable convex optimization problems originating in image processing and in location theory.

\subsection{Image deblurring and denoising}\label{subsec51}

For a given matrix $A \in \mathbb{R}^{m \times m}$ describing a \textit{blur operator} and a given vector $b \in \R^m$ representing the \textit{blurred and noisy image} the task that we considered was to estimate the \textit{unknown original image} $x^*\in\R^m$ fulfilling
$$Ax=b.$$
With this respect we dealt with the regularized least squares problems
\begin{align*}
	\hspace{-1.8cm}(P_2) \quad \quad \inf_{x \in \R^m}{\left\{\lambda \left\| x \right\|_1 +  \left\| Ax-b \right\|^2\right\}}
\end{align*}
and
\begin{align*}
	\hspace{-1.8cm}(P_3) \quad \quad \inf_{x \in \R^m}{\left\{\lambda \left\| x \right\|_1  +  \left\| Ax-b \right\|^2 + \delta_S(x)\right\}},
\end{align*}
where $S\subseteq \R^m$ is an $m$-dimensional cube representing the range of the pixels and $\lambda > 0$ the regularization parameter. One of our aims was to show that in some concrete cases the quality of the recovered image via classical $l_1$ regularized problem is by far not as good as the one recovered when regularizing with $\lambda\|\cdot\|_1 + \delta_S$. We solved problem $(P_2)$ by using Algorithm \ref{alg-2-functions} and problem $(P_3)$ by using Algorithm \ref{alg-k-f} and showed the benefits of having the first one extended to problems having in their objective the sum of more than two functions.

We concretely looked at the $272 \times 329$  \textit{blobs test image}, which is part of the image processing toolbox in Matlab. We scaled the pixels to the interval $[0,1]$ and vectorized the image, obtaining a vector of dimension $m=272 \times 329=89488$. Further, by making use of the Matlab functions {\ttfamily imfilter} and {\ttfamily fspecial}, we blurred the image as follows:
\begin{lstlisting}[numbers=left,numberstyle=\tiny,frame=tlrb,showstringspaces=false]
H=fspecial('gaussian',9,4);   % gaussian blur of size 9 times 9
                              % and standard deviation 4
B=imfilter(X,H,'conv','symmetric');  % B=observed blurred image
                                     % X=original image		
\end{lstlisting}
In row $1$ the function {\ttfamily fspecial} returns a rotationally symmetric Gaussian lowpass filter of size $9 \times 9$ with standard deviation $4$. The entries of $H$ are nonnegative and their sum adds up to $1$. In row $3$ the function {\ttfamily imfilter} convolves the filter $H$  with the image $X$ and outputs the blurred image $B$. The boundary option "symmetric" corresponds to reflexive boundary conditions.

Thanks to the rotationally symmetric filter $H$, the linear operator $A\in\mathbb{R}^{m \times m}$ given by the Matlab function {\ttfamily imfilter} is symmetric, too. By making use of the real spectral decomposition of $A$ it shows that $\left\| A \right\|^2=1$. After adding a zero-mean white Gaussian noise with standard deviation $10^{-3}$, we obtained the blurred and noisy image $b \in \mathbb{R}^m$ which is shown in Figure \ref{fig:blobs}.
\begin{figure}[ht]
	\centering
	\includegraphics*[viewport= 91 322 541 500, width=0.9\textwidth]{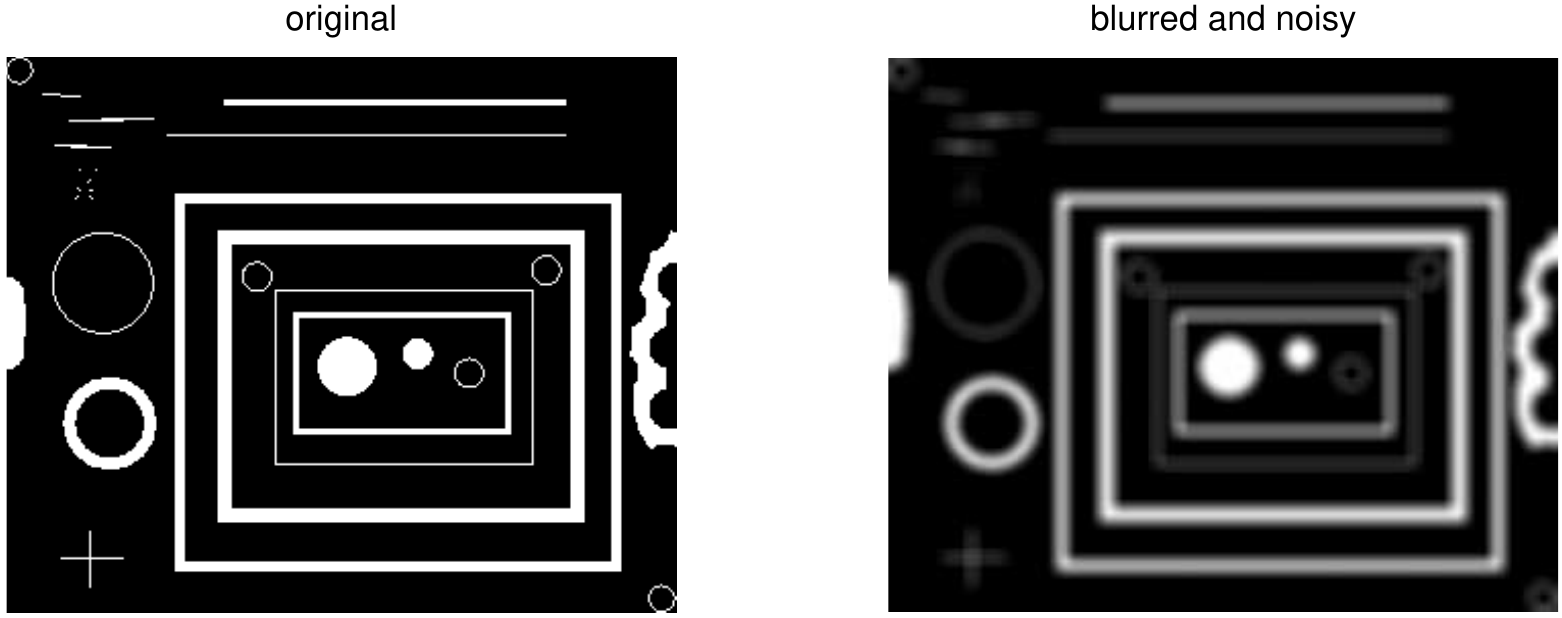}
	\caption{The $272 \times 329$ blobs test image}
	\label{fig:blobs}
\end{figure}
We solved the problem $(P_2)$ by applying Algorithm \ref{alg-2-functions} for $\lambda=2e-6$, $\sigma=0.01$, $\tau=9.99$, $f: \R^m \rightarrow \R, f(x) = \lambda\|x\|_1$, $g:\R^m \rightarrow \R$, $g(y) = \|y-b\|^2$ and $K=A$. Since $g^*(y) = \tfrac{1}{4}\|y\|^2 + \langle y, b\rangle$ for $y \in \R^m$, for all $z \in \R^m$ it holds
$$\prox\nolimits_{\sigma g^*}(z) = \argmin_{y \in \R^m} \left\{ \tfrac{\sigma}{4}\|y\|^2 + \sigma \langle y, b\rangle + \tfrac{1}{2}\|z-y\|^2\right \} = \tfrac{2}{\sigma + 2}(z-\sigma b),$$
while
$$\prox\nolimits_{\tau f}(z)_i = \argmin_{y_i \in \R} \left\{\tau\lambda |y|_i  + \tfrac{1}{2} |z_i-y_i|^2\right \} = \max \{|z_i|-\tau \lambda,0\}\sn(z_i) \ \forall i=1,...,m.$$
We solved the problem $(P_3)$ by applying Algorithm \ref{alg-k-f} for $k=3$, $\omega_i=\tfrac{1}{3}$, $i=1,2,3$, $\lambda=2e-6$, $\sigma=0.05$, $\tau=6.66$, $g_1: \R^m \rightarrow \R, g_1(x) = \lambda\|x\|_1$, $K_1 = \id_{\R^m}$, $g_2:\R^m \rightarrow \R$, $g_2(y) = \|y-b\|^2$, $K_2=A$ and $g_3 : \R^m \rightarrow \overline \R, g_3(x) = \delta_S(x),  K_3= \id_{\R^m}$. For all $z \in \R^m$ it holds
$$\prox\nolimits_{\sigma g_1^*}(z)_i = z_i - \sigma \prox\nolimits_{\tfrac{\lambda}{\sigma}\|\cdot\|_1}\left(\tfrac{1}{\sigma}z\right)_i = z_i - \max \{|z_i|-\lambda,0\}\sn(z_i) \ \forall i=1,...,m,$$
$$\prox\nolimits_{\sigma g_2^*}(z) = \tfrac{2}{\sigma + 2}(z-\sigma b)$$
and
$$\prox\nolimits_{\sigma g_3^*}(z) = z - \sigma \prox\nolimits_{\sigma\delta_S}\left(\tfrac{1}{\sigma}z\right) = z - \sigma P_S\left(\tfrac{1}{\sigma}z\right).$$
\begin{figure}[ht]	
	\centering
	\includegraphics*[viewport= 80 488 521 785, width=1\textwidth]{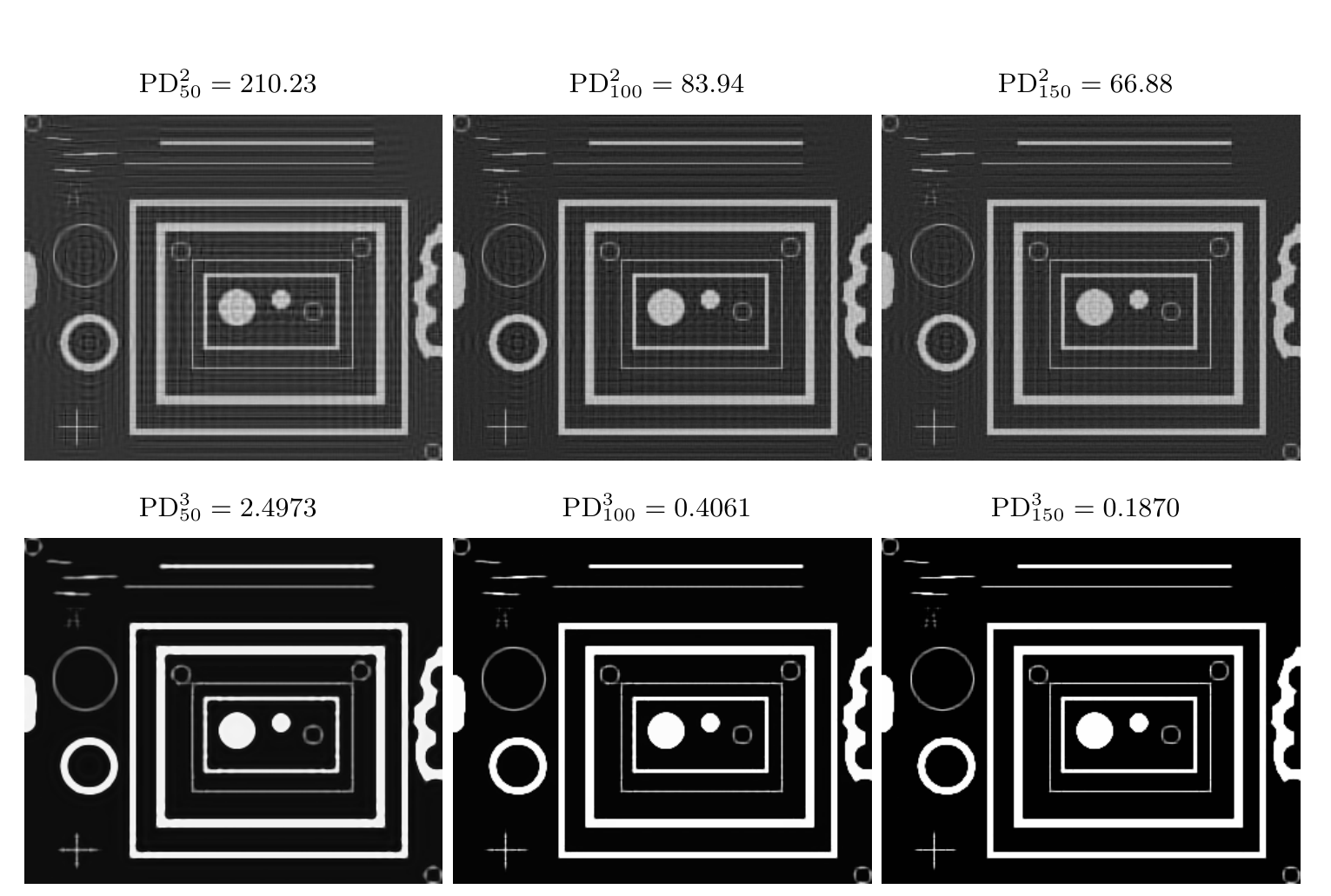}
	\caption{Iterations 50, 100 and 150 for solving $(P_2)$ via Algorithm \ref{alg-2-functions} and $(P_3)$ via Algorithm \ref{alg-k-f}}
	\label{fig:comparison}	
\end{figure}

The top line of Figure \ref{fig:comparison} shows the iterations 50, 100 and 150 of Algorithm \ref{alg-2-functions} for solving $(P_2)$, while the bottom line of it shows the iterations 50, 100 and 150 of Algorithm \ref{alg-k-f} for solving $(P_3)$, for each of them the value of the objective function at the respective iterate being provided. All in all the quality of the recovered image by solving $(P_3)$ significantly outperformed the one of the image recovered by solving the classical $l_1$ regularized least squares problem $(P_2)$. Moreover, in the images recovered by solving $(P_2)$ some artefacts could be identified. The gap between the quality of the recovered images is emphasized also by the \textit{improvement in signal-to-noise ratio (ISNR)}, which is defined as
$$ \text{ISNR}(n) = 10 \log_{10}\left( \frac{\left\|x-b\right\|^2}{\left\|x-x^n\right\|^2} \right),$$
where $x$, $b$ and $x^n$ denote the original, observed and estimated image at iteration $n$, respectively. Figure \ref{fig:ISNR} shows the evolution of the ISNR values when solving $(P_2)$ and $(P_3)$.
\begin{figure}[ht]	
	\centering
	\includegraphics*[viewport= 110 355 490 563, width=0.9\textwidth]{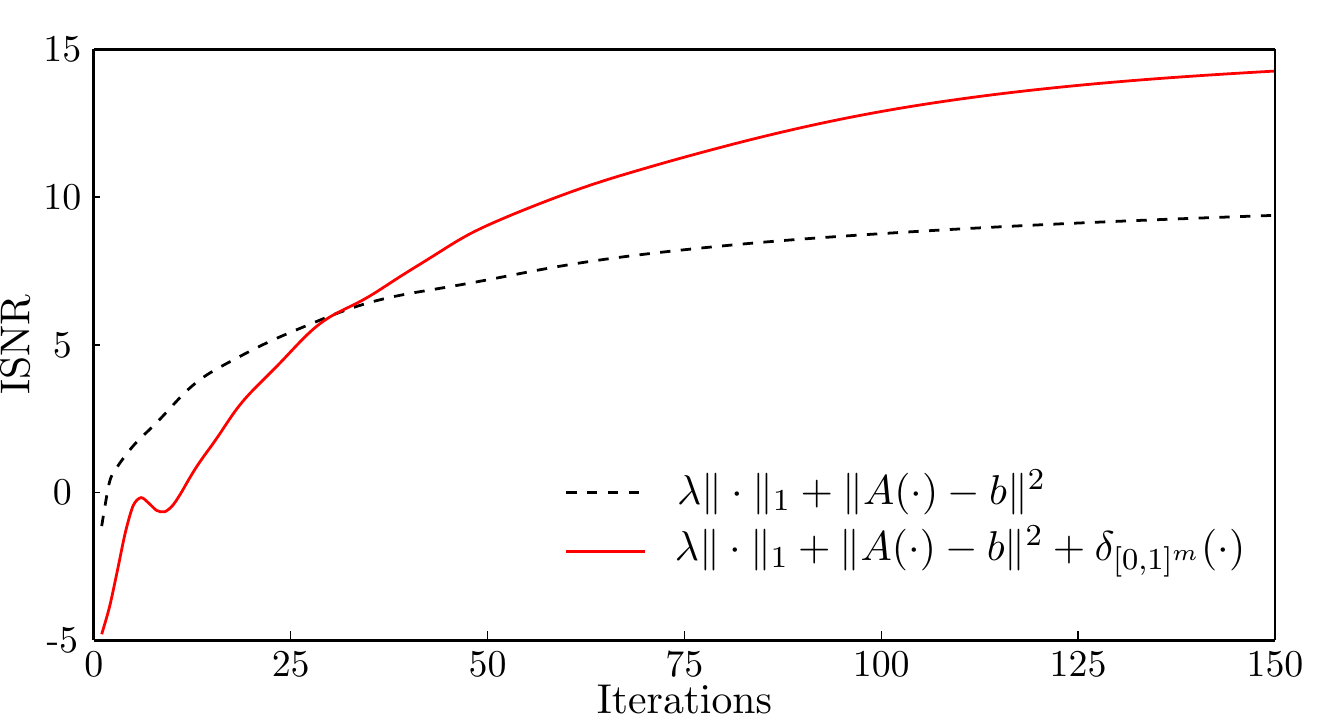}
	\caption{Improvement in signal-to-noise ratio (ISNR)}
	\label{fig:ISNR}	
\end{figure}

\subsection{The Fermat-Weber problem}\label{subsec52}

The second application of the primal-dual algorithm presented in this paper is with respect to the solving of the Fermat-Weber problem, which concerns the
finding of a new facility in order to minimize the sum of weighted distances to a set of fixed points. We considered the nondifferentiable convex optimization problem
\begin{align*}
	\hspace{-1.8cm}(P_{FW}) \quad \quad \inf_{x \in \R^m}\left \{\sum_{i=1}^k \lambda_i\|x-c_i\|\right\},
\end{align*}
where $c_i \in \R^m$ are given points and $\lambda_i > 0$ are given weights for $i=1,...,k$. We solved the optimization problem $(P_{FW})$ by using Algorithm
\ref{alg-k-f-id1} for $\omega_i = \tfrac{1}{k}$ and $g_i : \R^m \rightarrow \R, g_i(x) = \lambda_i\|x-c_i\|, i=1,...,k$. With this respect we used that for $i=1,...,k$ it holds
$$g_i^*(y) = \left \{
\begin{array}{rl}
\langle y,c_i \rangle, & \ \mbox{if} \ \|y\| \leq \lambda_i,\\
+\infty, & \ \mbox{otherwise},
\end{array}
\right. \ \forall y \in \R^m$$
and, from here, when $\sigma > 0$, that
$$\prox\nolimits_{\sigma g_i^*}(z) =  \left \{
\begin{array}{rl}
z-\sigma c_i, & \ \mbox{if} \ \|z-\sigma c_i\| \leq \lambda_i,\\
\lambda_i \frac{z-\sigma c_i}{\|z- \sigma c_i\|}, & \ \mbox{otherwise}
\end{array}
\right. \ \forall z \in \R^m.$$
We investigated the functionality of the algorithm on two prominent sets of points and weights, often considered in the literature when analyzing the performances of iterative schemes for the Fermat-Weber problem.

In a first instance we considered for $k=4$ the points in the plane and the weights
\begin{equation}\label{exampleFW1}
c_1=(59,0), c_2=(20,0), c_3=(-20,48), c_4=(-20,-48) \ \mbox{and} \ \lambda_1=\lambda_2=5, \lambda_3 = \lambda_4 = 13,
\end{equation}
respectively.  The optimal location point is $\widehat x = (0,0)$, however, the classical Weiszfeld algorithm (see \cite{lmw, weiszfeld}) with starting point $x^0=(44,0)$ breaks down in $(20,0)$. On the other hand, Algorithm \ref{alg-k-f-id1} with $\sigma=0.13$, $\tau =1.4$ and $y^0_k=(0,0), k=1,...,4$, achieved a point which is optimal up to three decimal points after $30$ iterations. Figure \ref{fig:bilderlocation1} shows the progression of the iterations, while $PD_n$ provides the value of the objective function at iteration $n$.
\begin{figure}[ht]	
	\centering
	\includegraphics*[viewport= 107 283 491 572, width=1\textwidth]{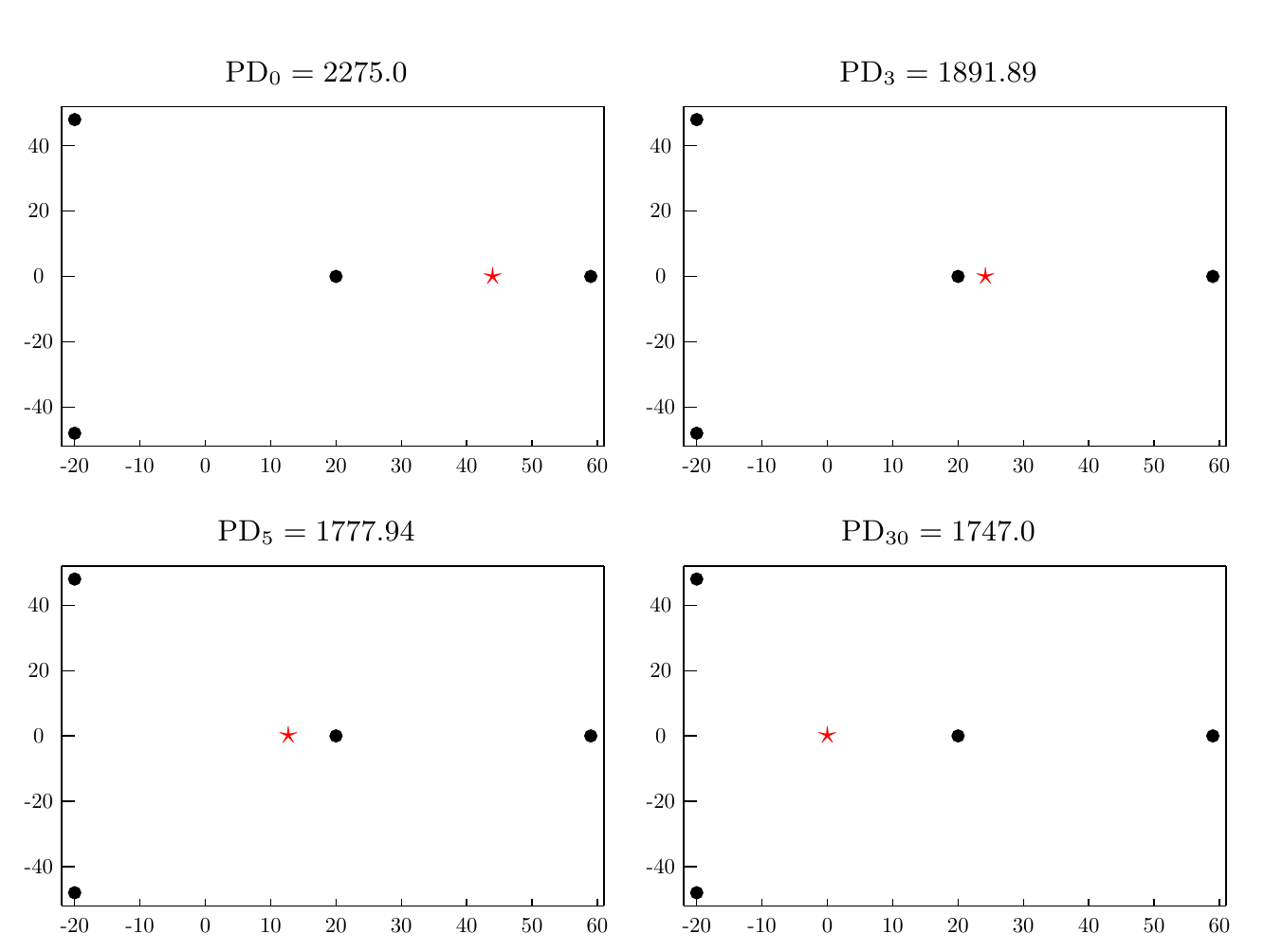}
	\caption{The progression of iterations of Algorithm \ref{alg-k-f-id1} when solving the Fermat-Weber problem for points and weights given by \eqref{exampleFW1}}
	\label{fig:bilderlocation1}	
\end{figure}
Recently an approach for solving the Fermat-Weber problem was proposed by Goldfarb and Ma in \cite{golma}, which assumes the approximation of each of the functions in the objective by a convex and differentiable function with Lipschitz-continuous gradient. The optimization problem which this smoothing method yields is solved in \cite{golma} by the classical gradient method (Grad) and by a variant of Nesterov's accelerated gradient method (Nest) (see \cite{nes}) and by a fast multiple-splitting algorithm (FaMSA-s) introduced in this paper. We applied the smoothing approach in connection with these algorithms to the example considered in \eqref{exampleFW1} with smoothness parameter $\rho$ equal to $10^{-3}$ (chosen also in \cite{golma}) and step sizes $\tau=0.1$, $\tau=0.01$ and $\tau=0.001$. We stopped the three algorithms when achieving an iterate $x^n$ such that $\|x^n - \widehat x\| \leq 10^{-3}$ and obtained in all cases the lowest number of iterations for $\tau=0.1$. A point which is optimal up to three decimal points was obtained for Nest after 308 iterations, for Grad after 175 iterations and for FaMSA-s after 54 iterations, none of these iterative schemes attaining the performance of Algorithm \ref{alg-k-f-id1}.

For the second example of the Fermat-Weber problem we considered in case $k=5$ the points in the plane and the weights (see \cite{drez})
\begin{equation}\label{exampleFW2}
\begin{aligned}
c_1=(0,0), c_2=(1,0), c_3=(0,1), & \ c_4=(1,1), c_5=(100,100)\\
\mbox{and} \ \lambda_1=\lambda_2=\lambda_3 = \lambda_4 = 1, & \ \lambda_5=4,
\end{aligned}
\end{equation}
respectively.  The optimal location point is $\widehat x = (100,100)$ and, by choosing the relative center of gravity $x^0=(50.25,50.25)$ as starting point,
we found out that not only the classical Weiszfeld algorithm, but also the approach from \cite{golma} described above in connection to each of the methods Grad, Nest and FaMSA-s did not achieve a point which is optimal up to three decimal points after millions of iterations. On the other hand, Algorithm \ref{alg-k-f-id1} with $\sigma=0.0001$, $\tau =9999$ and $y^0_k=(0,0), k=1,...,5$, achieved a point which is optimal up to three decimal points after $478$ iterations. This example is more than illustrative for the performance of the primal-dual Algorithm \ref{alg-k-f-id1} in comparison to some classical and recent algorithms designed for the Fermat-Weber problem. Figure \ref{fig:bilderlocation2} shows the progression of the iterations, while $PD_n$ provides the value of the objective function at iteration $n$.
\begin{figure}[ht]	
	\centering
	\includegraphics*[viewport= 158 254 449 575, width=0.7\textwidth]{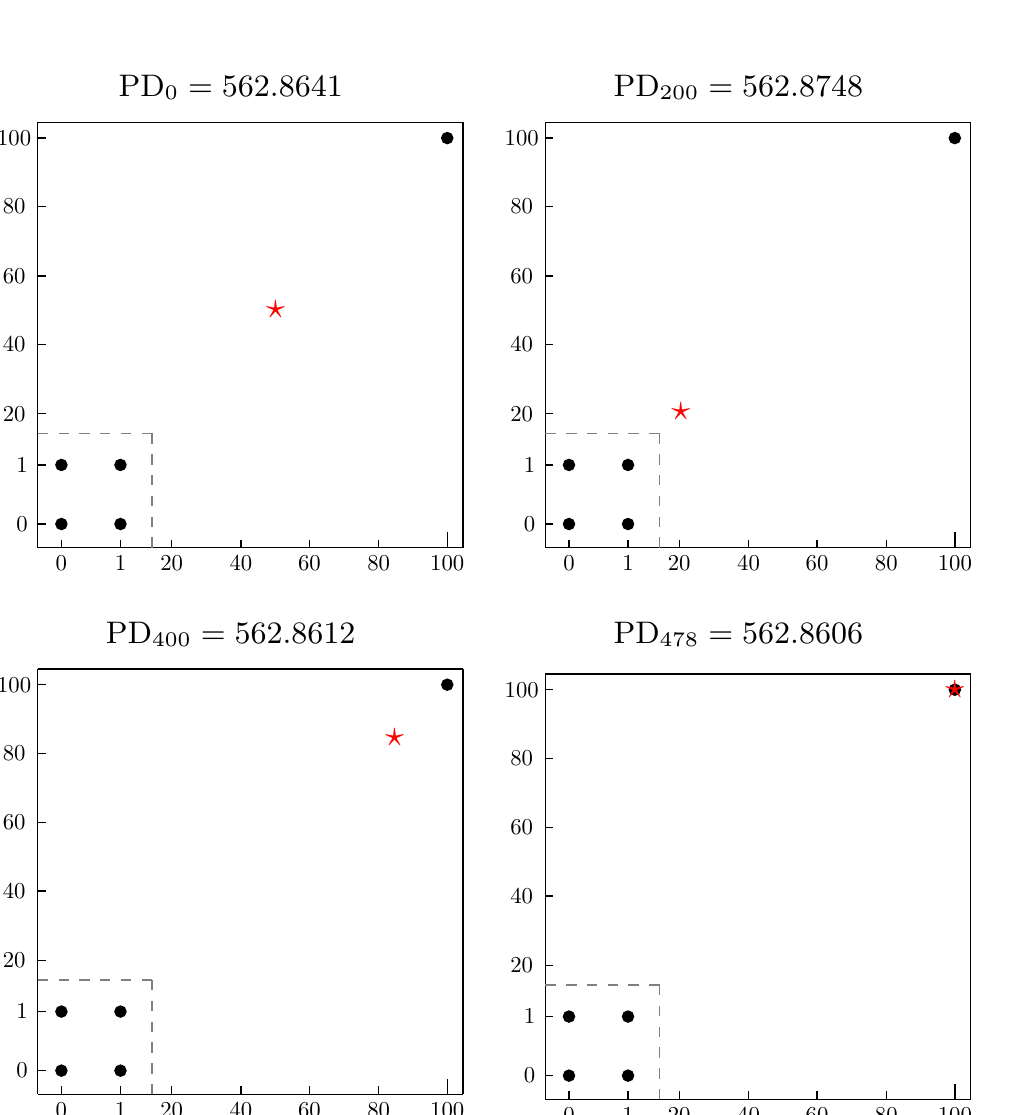}
	\caption{The progression of iterations of Algorithm \ref{alg-k-f-id1} when solving the Fermat-Weber problem for points and weights given by \eqref{exampleFW2}}
	\label{fig:bilderlocation2}	
\end{figure}

\section{Conclusions}\label{sec6}
In this paper we motivate and formulate a primal-dual algorithm which solves both the problem of finding the zeros of the sum of a maximally monotone operator with the composition of another maximally monotone operator with a linear continuous operator and its Attouch-Th\'era-type dual inclusion problem in Hilbert spaces. We also investigate the convergence of the provided iterative scheme and show how one can derive from it a splitting algorithm for finding the zeros of the sum of compositions of maximally monotone operators with linear continuous operators. As particular instances of the general schemes algorithms for solving several classes of nondifferentiable convex optimization problems are introduced. Among them one can rediscover the primal-dual algorithm from \cite{ch-pck} for solving the problem which assumes the minimization of the sum of a proper, convex and lower semicontinuous function with the composition of another proper, convex and lower semicontinuous function with a linear continuous operator. The performances of the provided algorithm are emphasized in the context of some applications in image deblurring and denoising and in location theory.\vspace{1ex}

\noindent{\bf Acknowledgements.} The authors are thankful to Christopher Hendrich for the implementation of the numerical schemes to which comparisons of the primal-dual algorithm were made.

\end{document}